\documentclass[11pt]{amsart}
\usepackage{amssymb, amsthm, epsfig, verbatim, amscd, enumerate, stmaryrd,
amsmath, psfrag}

\newtheorem{thm}{Theorem}[section]
\newtheorem{cor}[thm]{Corollary}
\newtheorem{lem}[thm]{Lemma}
\newtheorem{prop}[thm]{Proposition}

\theoremstyle{definition}
\newtheorem{defn}[thm]{Definition}
\newtheorem{conj}[thm]{Conjecture}

\newtheorem{rem}[thm]{Remark}
\theoremstyle{remark}

\numberwithin{equation}{section}

\newcommand{\al}{\alpha}
\newcommand{\be}{\beta}
\newcommand{\ga}{\gamma}

\newcommand{\de}{\delta}
\newcommand{\ep}{\varepsilon}

\newcommand{\ka}{\kappa}

\newcommand{\la}{\lambda}

\newcommand{\om}{\omega}

\newcommand{\Si}{\Sigma}

\newcommand{\Th}{\Theta}

\newcommand{\va}{\varphi}
\newcommand{\csi}{\xi}

\newcommand{\x}{\times}
\newcommand{\+}{\oplus}

\newcommand{\Z}{\mathbb Z}

\newcommand{\Q}{\mathbb Q}

\newcommand{\R}{\mathbb R}

\newcommand{\RP}{{\mathbb R}{P}}
\newcommand{\CP}{{\mathbb C}{P}}

\newcommand{\del}{\partial}

\newcommand{\co}{\colon\thinspace} 

\newcommand{\bls}[1]
{\operatorname{Bl}_{#1}}

\newcommand{\im}{\operatorname{im}}
\newcommand{\rank}{\operatorname{rank}}
\newcommand{\coker}{\operatorname{coker}}
\newcommand{\KO}{{K}_{\R}}
\newcommand{\rKO}{\widetilde {K}_{\R}}
\newcommand{\PD}{\operatorname{PD}}

\begin{document}
\mathsurround=1pt 
\title[Characteristic classes and existence of  singular maps]
{Relations among  characteristic classes and existence of  singular maps}



\author{Boldizs\'ar Kalm\'ar}
\address{
Alfr\'ed R\'enyi Institute of Mathematics,
Re\'altanoda u. 13-15, 1053 Budapest, Hungary}
\email{bkalmar@renyi.hu}
\thanks{The first author has been  partially supported by Magyary Zolt\'an Postdoctoral Fellowship
and  OTKA grant NK81203.}

\author{Tam\'as Terpai}
\address{
Alfr\'ed R\'enyi Institute of Mathematics,
Re\'altanoda u. 13-15, 1053 Budapest, Hungary}
\email{terpai@renyi.hu}
\thanks{The second author has been supported by OTKA grant NK81203.}



\begin{abstract}
We obtain  relations among the characteristic classes of a manifold $M$ 
admitting corank one maps.
Our relations yield strong restrictions on 
the cobordism class of $M$ and also
 nonexistence results for singular maps of 
the projective spaces.
We obtain our results through blowing up 
a manifold along the singular set of a smooth map and 
perturbing the arising non-generic corank one map.
\end{abstract}

\maketitle

\section{Introduction}

Let $M$ be a smooth closed
$n$-dimensional  manifold. 
In   \cite{Bott} it is shown that a subbundle $\csi$ of the tangent bundle $TM$
 is  tangent  to the leaves of a smooth foliation of $M$ (that is, $\csi$  is integrable) only if
the ring generated by the real Pontryagin classes of $TM/\csi$
vanishes in dimensions greater than $2(n -  \dim \csi)$.
The primary purpose of our paper 
is to  prove analogous vanishing theorems
about the Stiefel-Whitney and rational Pontryagin classes 
 in the case of ``smooth singular fibrations", i.e.\ singular maps of $M$. 
For $n > k \geq 0$ let
$Q$ be  a smooth $(n-k)$-dimensional manifold
and
let $f \co M \to Q$ be a smooth map. Denote by $\Si$ the set of
singular points of $f$. A point $p \in \Si$ is a $\Sigma^{i_1}$ singularity
of $f$, in notation $p \in \Sigma^{i_1}$, if the rank of the differential
$df$ is equal to $n-i_1$ at $p$. Inductively we define $\Sigma^{i_1, \ldots,
i_r} \subset M$, where $i_1 \geq \cdots \geq i_r \geq 0$, by taking the
$\Sigma^{i_r}$ points of the restriction $f |_{\Sigma^{i_1,\ldots,
i_{r-1}}}$\footnote{After a generic perturbation of $f$, we can assume
that $\Sigma^{i_1,\ldots, i_{r-1}}$ is a smooth submanifold of $M$, see
\cite{Bo}.}.
A {\it Morin map} is a smooth map with only 
$\Sigma^{k+1,1,\ldots,1,0}$ singularities (also called $A_m$-singularities,
where $m-1$ is the number of copies of ``$1$''). 
In the present paper, we show that 
the existence of a Morin map from  $M$ to $Q$
implies the vanishing of the ideal generated by the differences  
$w_I(TM) - w_J(TM) \in H^*(TM; \Z_2)$ of monomials of the same degree
consisting of Stiefel-Whitney 
classes of  sufficiently high degrees,
 where $I$ and $J$ run over all the multiindices with length $|I| = |J|$.
 In particular, we have

\begin{thm}\label{bevclass}
Let  $k$ be odd, 
 $M^n$ be orientable, and suppose there exists a Morin map $f \co
M \to \R^{n-k}$. Then
$$\prod_{j=1}^m w_{r_j}(TM) = \prod_{j=1}^m w_{s_j}(TM)$$
for any $m$ and collections $r_j, s_j$, $j=1, \dots, m$, which satisfy the
conditions $r_j, s_j \geq k+3$, $j=1, \ldots, m$, and $\sum_{j=1}^m r_j =
\sum_{j=1}^m s_j$.
 If $f$ is a fold map, then the same holds with $r_j, s_j \geq k+2$, $j=1,
\ldots, m$.
\end{thm}
 
We prove this in a more general form, see Theorem~\ref{charclass}.
 In the proof we proceed by blowing up
 the source manifold of a Morin map $f$ along
the singular set and perturbing $f \circ \pi$, where $\pi$ is the
projection map of the blowup, see Theorems~\ref{perturb} and \ref{perkov}.

As an application of Theorem~\ref{bevclass}
and Proposition~\ref{RPfold} , we obtain
 
\begin{thm}
Let $n=2^d+c$ with $0\leq c < 2^d-2$, $c$ is odd.
\begin{enumerate}[\rm (1)]
\item
There is no fold map of $\RP^n$ into $\R^{n-k}$
for 
$1 \leq k+1 < c$.
\item
There is no Morin map of $\RP^n$ into $\R^{n-k}$
for 
$k+2 < c$ if $k$ is odd.
\end{enumerate}
\end{thm}

For example there is no Morin map of $\RP^{13}$ into  $\R^{12}$,
and there is no fold map of $\RP^{11}$ into $\R^{10}$.
 
We call  a smooth map  from $M^n$ to $Q^{n-k}$
 a {\it corank $1$  map} if
the rank of its differential
is not less than $n-k-1$ at any point of $M$.
 About the vanishing of rational Pontryagin classes of $TM$, we have the analogous result to \cite{Bott}:

 \begin{thm}
 Suppose $M^n$  admits a corank $1$ map into $\R^{n-k}$.
Then
  the rational Pontryagin class $p_i^{\Q}(TM) \in H^{4i}(M;\Q)$ vanishes for $2i >
k+1$.
 \end{thm}
 
 For example, there is no corank $1$ map of $\CP^n$ into $\R^{2n-k}$
if $\lfloor n/2 \rfloor \geq (k+2)/2$.
By Thom transversality and computing the codimension of the Boardman manifolds \cite{Bo},
we have that if $n < 2(k+2)$, then $M^n$ admits 
 corank $1$ maps  into $Q^{n-k}$. 
 
 Hence for even $n$, we obtain that 
$\CP^n$ has a corank $1$ map into $\R^{2n-k}$
if and only if $n < k+2$.
 For odd $n \geq 3$, we do not know whether  corank $1$ maps
 exist
 from $\CP^n$ to $\R^{n+2}$.

 We also obtain results about the cobordism class of the source manifold of a Morin map by
 combining our relations among characteristic numbers of
 the source manifold (see Proposition~\ref{karosztalyok})
 with Dold relations.

\begin{thm}
Suppose $M^n$ is orientable and admits a fold map into $\R^{n-k}$. Then
\begin{enumerate}[\rm (1)]
 \item
if $k=1$ and $0< n  \neq 2^a + 2^b -1$, $a > b \geq 0$, then
 ${M}$ is null-cobordant,
 \item
if $n>k \geq 5$, $k$ is odd, $k \neq 2^a - 1$, $a \geq 3$, $n-k \neq 1,3,7$ and
 $w_i(TM)=0$ for $i=2, \ldots,  k$, then
${M}$ is null-cobordant.
\end{enumerate}
\end{thm}

For fold maps into $(n-k)$-dimensional manifolds with $k=2^a-1$, $a >1$, we have 
Conjecture~\ref{nullcobnagykodim}, which we 
verified for $n \leq 1200$ and $3 \leq k \leq 1023$ by using a computer.

\begin{thm}
Suppose $M^n$ is orientable and admits a Morin map into $\R^{n-k}$.
 If $n-k = 5,6$ or $n-k \geq 9$, $k$ is odd and 
 $w_i(TM)=0$ for $i=2, \ldots,  k+1$, then $M$ is null-cobordant.
\end{thm}

Note that $w_j(TM)=0$ holds for all $j=1, \dots, k$ if for example $M$ is
$k$-connected, i.e.\ all the homotopy groups $\pi_j(M)$ vanish for $1 \leq j
\leq k$.

Our results give  
easy to use 
 criteria
for the existence of fold maps, Morin maps and corank $1$ maps in general.
Up to the present, some  practical  
methods to check the existence of some singular map in general
have already been obtained:
\begin{itemize}
\item
There exists a fold map $f \co M \to Q$ with
  cokernel $f^*TQ / f^*df(TM)$   being trivial  on the singular set
 if and only if there is a bundle epimorphism $TM \+\ep^1
\to TQ$ \cite{An, Sa1}. This gives a complete answer to
the problem of existence of fold maps with $k \equiv 0 \mod 2$ 
\cite{An}, which can be easily used for further computations when $k$ is even, see
for example \cite{SadSaeSak}.
\item
More general versions of this result are deep theorems stating h-principles,
which are hard to apply directly
and led to criteria using Thom polynomials,
 see for example \cite{An01, An02, An0}.
 \item
There exist  fold maps and  cusp maps of $M$ into an almost parallelizable
manifold only if the Euler characteristic $\chi(M)$ is even, under the assumption
that $n-k$
is big enough \cite{SaeSak}. Refinements of \cite{SaeSak} 
include results for Morin maps as well when $k$
is odd \cite{An2, Sad} but nothing is known when $\chi(M)$ is even. 
\item
For odd $k$,
 the self-intersection class of the singular set of a generic
corank $1$ map $f$ of $M$ into $Q$
is equal  to the $(k+1)/{2}$-th Pontryagin class of 
$TM - f^*TQ$ modulo $2$-torsion \cite{OhmSaeSak}.
\end{itemize}

The paper is organized as follows. In \S2 we  present the main results about
blowing up the source manifold of a singular map. In \S3 we present the
main results about the characteristic classes of the source manifold of a
singular map. In \S4 we prove the statements of \S2, and in \S5 we prove the
statements of \S3.

\subsection*{Acknowledgements} 
The authors would like to thank Andr\'as Sz\H{u}cs for his advices, which improved
 the paper.

\subsection*{Conventions}

All manifolds henceforth are assumed to be smooth of class $C^{\infty}$.
Let $\mathfrak N_n$ denote the unoriented cobordism group of $n$-dimensional
manifolds. The term ``cobordant'' refers to unoriented cobordism unless 
oriented cobordism is specified  explicitly.
For a finite CW-complex $X$,
$\rKO(X)$ and ${\KO}(X)$  denote the reduced and unreduced  real K-rings of $X$, 
respectively, with $\rKO(X) \subseteq {\KO}(X)$. 
The symbol $\ep^n_X$
denotes the trivial $n$-dimensional bundle over the space $X$, the lower
index $``X"$ will be omitted when it is clear from the context. 
Wherever not stated
otherwise, we use the convention that if $\be < 0$ or $\al < \be$, then
the binomial coefficient $\binom{\al}{\be} = 0$.

\section{Blowing up the source manifold along the singular set}

For $n > k \geq 0$, let $M$ be a closed $n$-manifold and $Q$ be an $(n-k)$-manifold.
It is known that Morin maps are generic corank $1$ maps\footnote{In \cite[Chapter VI \S1]{GoGu}
it is called $1$-generic and corank at most $1$ everywhere.},
the singular set  of a Morin map of $M$ into $Q$
 is an embedded $(n-k-1)$-dimensional manifold,
the closure of $\Sigma^{k+1,1,\ldots,1,0}$ is $\Sigma^{k+1,1,\ldots,1}$ with the same number
of copies of ``$1$", and we have  $1$-codimensional embeddings
$\Sigma^{k+1} \supset \Sigma^{k+1, 1} \supset \cdots$ of
closed manifolds.
 Morin maps with only $\Sigma^{k+1, 0}$ 
and $\Sigma^{k+1,1, 0}$ singularities
are called cusp maps, while cusp maps 
with only $\Sigma^{k+1, 0}$
singularities are called fold maps.
Furthermore, the points of $\Sigma^{k+1, 0}$ and $\Sigma^{k+1, 1, 0}$ are called fold singular points and
cusp singular points, respectively.
We note that in general a corank $1$ map cannot be perturbed to obtain a  Morin map.

For an odd $k \geq 1$, let $f \co M^{n} \to Q^{n-k}$ be a Morin map. 
We denote the $(k+1)$-dimensional normal bundle of
its singular set $\Sigma = \Sigma^{k+1}$
by $\csi$. For $0 \leq \la \leq (k+1)/2$ let $\Sigma^{k+1,0}_\la $ be the set of
index $\la$ fold singular points\footnote{The index is well-defined
if we consider that $\la$ and $k+1-\la$ represent the same index.}
 of $f$. 
 Denote by $\eta$ the restriction of $\csi$ to $\Sigma^{k+1,0}_{(k+1)/2}$.
 Then, the normal bundle $\eta$ has structure group $G(\eta)$ generated by
transformations of the form
$$(x_1,\ldots, x_{k+1}) \mapsto A(x_1,\ldots, x_{k+1}) \text{ with } A\in O\left ( \frac{k+1}{2} \right )
\times O \left ( \frac{k+1}{2} \right )$$
and 
$$(x_1,\ldots, x_{k+1}) \mapsto (x_{(k+1)/2+1}, \ldots, x_{k+1}, x_1, \ldots,
x_{(k+1)/2}).$$
The restriction of $f$ to any fiber of $\eta$ is left-right equivalent to
the saddle singularity $$(x_1,\ldots, x_{k+1}) \mapsto \sum_{i=1}^{(k+1)/2} x_i^2-
\sum_{i=(k+1)/2+1}^{k+1} x_i^2,$$
i.e., to the fold singularity of index $(k+1)/2$.

Note that even if $M$ and $Q$ 
are
oriented, the index $(k+1)/2$ indefinite fold singular set of $f$ can be
non-orientable.

\begin{defn}[Blowup]\label{blowup}
Let $V$ be an $l$-dimensional closed submanifold of ${M^n}$,
and denote the $(n-l)$-dimensional normal bundle of $V$ by $\zeta$.
Let $\bls \zeta {M}$ denote the manifold obtained by blowing up $M$ along
$V$. Let $\bls \zeta f$ denote the composition $f \circ \pi$ where $\pi \co
\bls \zeta {M} \to M$ is the natural projection.
\end{defn}


\begin{rem}
Let $f \co M^n \to Q^{n-k}$ be a generic corank $1$ 
map. 
Then the singular set $\Si$ is an embedded $(n-k-1)$-dimensional submanifold of $M$ \cite{Bo}, 
denote its normal bundle  by $\csi$.
We have that the map
$\bls \csi f$ is a non-generic corank $1$ map and 
its singular set is $\pi^{-1}(\Si)$.
\end{rem}

We will use the notations of the following blowup diagram.

\begin{equation*}\label{invaridiag}
\begin{CD}
\pi^{-1}(\Si) @> \tilde  \imath >> \bls \zeta {M} \\
@VV p V @V \pi VV \\
\Si @> i >> {M}
\end{CD}
\end{equation*}

Note that $\pi^* \co H^{m}(M; \Z_2) \to H^{m}(\bls \zeta {M};\Z_2)$ is
injective for all $m$. Indeed, consider the Gysin map $\pi_! \co
H^{m}(\bls \zeta {M};\Z_2) \to H^{m}( M; \Z_2)$. Denote by $\PD$ the
Poincar\'e duality map $\PD \co H_n({M^n}; \Z_2) \to H^0({M^n}; \Z_2)$, then
for any $x \in H^m(M;\Z_2)$ we have
\begin{align*}
\pi_! (\pi^*(x)) &= \pi_!(\pi^*(x) \cup 1) = x \cup \pi_!(1) = x \cup \PD
(\pi_*([\bls \zeta {M}] \cap 1)) =\\
&=x \cup \PD (\pi_*([\bls \zeta {M}])) = x \cup \PD ([M]) = x \cup 1 =x.
\end{align*}

\begin{defn}[Morse-Bott map]\label{BottMorsedef}
For $n > k \geq 0$, we call a smooth map $f \co P^{n} \to Q^{n-k}$ a
{\it Morse-Bott map} if
\begin{enumerate} 
\item
the set $S_f$ of singular points of $f$ is the disjoint union $\sqcup_i S_i$
of smooth closed connected submanifolds of $P$,
\item
each component $S_i$ is the total space of a smooth bundle with a connected
manifold $C_i$ as fiber,
\item
for each component $S_i$ there exist $\la$ and $l$ such that $0 \leq \la
\leq l \leq k+1$ and for each singular point $p \in S_i$ there exist
neighborhoods $U_1$ of $p$, $U_2$ of $f(p)$ and diffeomorphisms $u_1 \co U_1
\to \R^l \x \R^{k+1-l} \x \R^{n-k-1}$ and $u_2 \co U_2 \to \R \x \R^{n-k-1}$
with the following properties:
\begin{enumerate}
\item
$u_1(p) = 0$, $u_2(f(p)) = 0$,
\item
$u_1(U_1 \cap S_f) = \{ (x,y,z) \in  \R^l \x \R^{k+1-l} \x \R^{n-k-1} : x = 0 \}$,
\item
for the fiber $C_i$ containing $p$,
$u_1(U_1 \cap C_i) = \{ (x,y,z) \in  \R^l \x \R^{k+1-l} \x \R^{n-k-1} : x = 0, z=0
\}$, and
\item
$u_2 \circ f \circ u_1^{-1}(x,y,z) =
(\sum_{i=1}^{\la}- x_i^2 + \sum_{i=\la+1}^{l} x_i^2, z)$.
\end{enumerate}
\end{enumerate}
The index of $f$ at a singular point $p$ is the pair $( \la, k+1 -l)$ if
$\la \leq l-\la$.
\end{defn}

\begin{rem}
Compare Definition~\ref{BottMorsedef} with \cite[Morse-Bott Lemma]{BaHu}.

\par 

Note that for a Morse-Bott map $f \co P \to Q$ the index is well-defined.
Let $\Si_{(\la, k+1-l)}$ denote the set of singular points of $f$ which
have index $(\la, k+1-l)$, then $\Si_{(\la, k+1-l)}$  is an $(n-l)$-dimensional closed submanifold
of $P$. Also note that a Morse-Bott map is a corank $1$ map, although it
is not necessarily  generic or Morin.

\par

For each index $(\la, k+1-l)$, let $\tilde \Si_{(\la, k+1-l)}$ denote the
set $\Si_{(\la, k+1-l)} /\sim$ where $p \sim q$ if and only if $p$ and
$q$ lie in the same connected fiber $C_i$ for some $i$.
Clearly $\tilde \Si_{(\la, k+1-l)}$ is an $(n-k-1)$-dimensional
manifold and the continuous map $\tilde f_{(\la, k+1-l)} \co \tilde
\Si_{(\la, k+1-l)} \to Q$ determined by the property $f = \tilde f_{(\la,
k+1-l)} \circ q_{\sim}$, where $q_{\sim} \co  \Si_{(\la, k+1-l)} \to
\tilde \Si_{(\la, k+1-l)}$ is the quotient map, is an immersion.

\par

The cokernel bundle $(f^*TQ / f^* df( TP))|_{S_f}$
can be identified with  the pull-back $$q_{\sim}^* (\cup_{(\la, k+1-l)} \tilde f_{(\la, k+1-l)}) ^* \nu,$$
where $\nu$ is the normal bundle of the immersion
$\cup_{(\la, k+1-l)} \tilde f_{(\la, k+1-l)} \co \cup_{(\la, k+1-l)} \tilde
\Si_{(\la, k+1-l)} \to Q$, where
$(\la, k+1-l)$ runs over all the indices of $f$.
\par

If $\la \neq l-\la$, then the normal bundle of the immersion $\tilde
f_{(\la, k+1-l)}$ is trivial.
\end{rem}

\begin{thm}\label{perturb}
Let $k \geq 1$ be odd.
For a fold map $f \co {M^n} \to Q^{n-k}$, we can perturb the map
$$\bls \eta f \co \bls \eta {M} \to Q$$
in a neighborhood of $\pi^{-1}(\Sigma^{k+1,0}_{(k+1)/2})$ so that the perturbed
map $\Th \co \bls \eta {M} \to Q$ is Morse-Bott 
and the normal bundle of the immersion
$\tilde
\Th_{(\la, k+1-l)}$ is trivial for each index $(\la, k+1-l)$.
The stable tangent bundle of $\bls \eta {M}$ splits as
$T{\bls \eta {M}} \oplus \ep^{1} \cong \zeta^{k+1} \oplus \Th^*TQ$
for some $(k+1)$-dimensional vector bundle $\zeta^{k+1}$ over $\bls \eta M$.
\end{thm}

By extending Theorem~\ref{perturb} to  Morin maps, we obtain

\begin{thm}\label{perkov}
Let $k \geq 1$ be odd.
Assume there exists a Morin map $f \co M^n \to
Q^{n-k}$.
 Then the stable tangent bundle of $\bls \csi {M}$ splits as
$$
T{\bls \csi {M}} \oplus \ep^{2} \cong \zeta^{k+2} \oplus (\widetilde {\bls \csi f})^*TQ$$
for some $(k+2)$-dimensional vector bundle $\zeta^{k+2}$ over $\bls \csi M$
and perturbation $\widetilde {\bls \csi f}$ of ${\bls \csi f}$.
\end{thm}

\begin{rem}
In Theorems~\ref{perturb} and \ref{perkov} 
if $Q$ is stably parallelizable, then 
the bundles 
$T{\bls \eta {M}}$ and
$T{\bls \csi {M}}$ 
are stably equivalent to  $(k+1)$- and $(k+2)$-dimensional bundles, respectively.
\end{rem}

\begin{rem}\label{foldpertkov}
If we
blow up $M$ along all the singular set $\Si$ of a fold map, then
we can perturb $\bls \csi f \co \bls \csi {M} \to Q$ in 
a neighborhood of $\pi^{-1}(\Sigma^{k+1,0}_{(k+1)/2})$ so that
the stable tangent bundle of $\bls \csi {M}$ splits as
$$
T{\bls \csi {M}} \oplus \ep^{1} \cong \zeta^{k+1} \oplus (\widetilde {\bls \csi f})^*TQ$$
for some $(k+1)$-dimensional vector bundle $\zeta^{k+1}$ over $\bls \csi M$
and perturbation $\widetilde {\bls \csi f}$ of ${\bls \csi f}$.
\end{rem}

\section{Characteristic classes of the source manifold}

By using the results of \S2, we obtain the following relations between the
Stiefel-Whitney classes of the source manifold of a Morin map.

\begin{thm}\label{charclass}
Let $k$ be odd,  $M$ be an orientable $n$-manifold and 
$Q$ be an orientable $(n-k)$-manifold.
Assume  $K \geq 0$ is such that $w_i(TQ) = 0$
for $i > K$,
furthermore for any $m$ and $j=1, \ldots, m$ let $r_j, s_j \geq k+3 + K$  be 
natural numbers such that $\sum_{j=1}^m r_j = \sum_{j=1}^m s_j$. If
$M$   admits a Morin map into $Q$, then
\begin{equation}\label{piequ}
 w_{r_1}(TM)\cdots w_{r_m}(TM) = 
 w_{s_1}(TM)\cdots w_{s_m}(TM).
\end{equation} 
The same holds under the relaxed condition $r_j, s_j \geq k+2 + K$ if there is
a fold map of $M$ into $Q$.
\end{thm}

 For example, $w_5(T\RP^{13})= 0$ and $w_6(T\RP^{13}) \neq 0$, thus there is no Morin map of
$\RP^{13}$ into $\R^{12}$.

\begin{rem}
\noindent
\begin{enumerate}[\rm (1)]
\item
Theorem~\ref{charclass} holds 
also if $M$ and $Q$ are possibly non-orientable and the Morin map of $M$ into $Q$
is a cusp map.
\item
Note that if $k \geq 0$ and $k$ is even, then  (\ref{piequ}) obviously
holds for a fold map if $r_j, s_j \geq k+2 + K$ since
$w_j(TM) = 0$ for $j \geq k+2+K$, see  \cite{An}.
\end{enumerate}
\end{rem}

\par

By Proposition~\ref{RPfold}  and applying the above to maps of the projective spaces $\RP^n$, we obtain

\begin{cor}
Let $n=2^d+m$ with $0\leq m < 2^d-2$, where $m$ is odd. 
There is no Morin map of $\RP^n$ into $\R^{n-k}$ for $k+2 < m$ if $k$
is odd.
There is no fold map of $\RP^n$ into $\R^{n-k}$ for $k+1 < m$.
\end{cor}
 For example, there is no fold map from $\RP^{13}$ to $\R^j$ with
$10 \leq j \leq 13$.

\begin{rem}
By \cite[Corollary~11.15]{MiSta}, we obtain the analogous result for
closed $n$-manifolds $M$ with $H^*(M;\Z_2) \cong H^*(\RP^n;\Z_2)$, and 
also a more general result
for any closed manifold whose cohomology ring 
with $\Z_2$ coefficients
is generated by one element.
\end{rem}

Now let us  consider Pontryagin classes. 

\begin{thm}\label{szingtusk}
Let $f \co M^n \to Q^{n-k}$, $n > k \geq 0$, be a smooth map with $\rank df
\geq n-k-1$ and let  $Q$ be  stably parallelizable. Then the rational
Pontryagin classes $p_i^{\Q}(TM) \in H^{4i}(M;\Q)$ vanish for $2i > k+1$.
\end{thm}

\begin{rem}\label{pontpeld}
For $n \geq 2$ 
the class $p_{\lfloor n/2 \rfloor}(T\CP^{n})$ is equal to $\binom{n+1}{{\lfloor
n/2 \rfloor}}y$, where $y$ is the standard generator of $H^{4\lfloor n/2
\rfloor}(\CP^n)$ and hence $p_{\lfloor n/2 \rfloor}^{\Q}(T\CP^{n})$ does not vanish.
Hence there is no corank $1$ map of $\CP^n$ to a stably parallelizable
target $Q^{2n-k}$ if $\lfloor n/2 \rfloor \geq (k+2)/2$.
For example, there exists no Morin map from $\CP^{2}$ to $Q^{4}$, and
from $\CP^{4}$ to $Q^{7}$ (cf.\  \cite[Example~4.9]{OhmSaeSak}) or to $Q^8$
(cf.\  \cite[Theorems~1.2 and 1.3]{SaeSak}), and there exists no Morin map
of $\CP^{49}$ to $Q^{93}$ (cf.\  \cite[Remark~4.5]{OhmSaeSak}).
\end{rem}

Finally, from the viewpoint of K-theory
 and  $\ga$ operations, we have the following\footnote{We presented these results at the conference ``Singularity
theory of smooth maps and related geometry'', RIMS, Tokyo, and
on the Topology Seminar at Kyushu University, Fukuoka, in 2009 December.}.
Recall that for a finite CW-complex $X$ the geometric dimension $g.dim(x)$
of an element $x \in \rKO(X)$ is the least integer $k$ such that $x+k$ is a
class of a genuine vector bundle over $X$ (see e.g. \cite{At}).

\par

We call a  corank $1$ map  $f \co M \to Q$ {\it tame}
if 
the $1$-dimensional cokernel bundle $\coker df|_\Si$
 of 
 the restriction $df|_\Si \co TM|_\Si \to f^*TQ$ is trivial.
 For example,
every fold map is tame for  $k \equiv 0 \mod{2}$ \cite {An} and
it is easy to construct not tame fold maps for  odd $k \leq n-3$, even between
orientable manifolds.
Also note that
a Morse-Bott map $f$ is tame if and only if all
the normal bundles of the immersions $\tilde f_{(\la, k+1-l)}$ are trivial.

Let $M^n$ and $Q^{n-k}$ be a closed $n$-manifold and an 
$(n-k)$-manifold, respectively.
\begin{prop}\label{Kgamma}
The following are equivalent:
\begin{enumerate}[\rm (1)]
\item
$M$ admits a tame corank $1$ map into $Q$,
\item
there is a fiberwise epimorphism $TM \+ \ep^1 \to TQ$.
\end{enumerate}
If $Q$ is stably parallelizable, then {\rm {(1)}} and {\rm {(2)}} hold if and only if 
$g.dim([TM] - [\ep^n]) \leq k+1$.
\end{prop}
For a finite CW-complex $X$, let $\la_t = \sum_{i=0}^{\infty} \la^i t^i$,
where $\la^i$ are the exterior power operators (for details, see \cite{At}).
Define $\ga_t = \sum_{i=0}^{\infty} \ga^i t^i$ to be the homomorphism
$\la_{t/1-t}$ of ${\KO}(X)$ into the multiplicative group of formal power
series in $t$ with coefficients in ${\KO}(X)$ and constant term $1$.
By the above proposition and \cite[Proposition~2.3]{At}, we immediately have
\begin{cor}\label{gam}\footnote{Compare with \cite[Proposition~3.2]{At}.}
If $M^n$ admits a tame corank $1$ map into a stably parallelizable $Q^{n-k}$, then
\begin{enumerate}[\rm (1)]
\item
$w_i(TM) = 0$ for $i \geq k+2$,
\item
$p_i(TM) = 0$ for $2i > k+1$,
\item
$\ga^i([TM] - [\ep^n]) = 0$ for $i \geq k+2$.
\end{enumerate}
\end{cor}
\begin{rem}\label{kettohatv}
Note that  the conditions (1) and (2) may not give strong results in general: for
example, all the positive degree Stiefel-Whitney and Pontryagin classes of
$\RP^{2^n-1}$ vanish\footnote{We have $w(T\RP^{2^n-1}) = (1+x)^{2^n} = 1 \in
\Z_2[x]/x^{2^n} = H^*(\RP^{2^n-1}; \Z_2)$, where $x$ denotes the generator
of $H^1(\RP^{2^n-1}; \Z_2)$. The natural homomorphism $H^s(\RP^{2^n-1}; \Z)
\to H^s(\RP^{2^n-1}; \Z_2)$ is an isomorphism for all positive even $s$. Our
claim follows by applying the fact that $p_i \equiv w_{2i}^2 \mod{2}$.}, and
if $k+1 \geq n/2$, then condition (2) is satisfied trivially for any $M$. 
In particular cases, though, condition (1) can still give strong results,
e.g.\ all Stiefel-Whitney classes of $\RP^{2^n-2}$ of degree up to $2^n-2$
are nonzero. 
\end{rem}

For an integer $s$ let $2^{R(s)}$ be the maximal power of $2$ which divides
$s$, and define $\ka(n) = \max \{ 0 < s < 2^{n-1} : s - R(s)< 2^{n-1} -n \}$.
By using Corollary~\ref{gam} (3) and following a similar argument
 to \cite{At}, we obtain the following:
\begin{prop}\label{alkgamma}
For $n \geq 4$, $\RP^{2^n-1}$ does not admit tame corank $1$ maps into
$\R^{2^n-1-k}$ for $k \leq {\ka(n)-2}$.
\end{prop}
\begin{rem}
Obviously
$s_0= 2^{n-1}-2^{\min \{ r : r + 2^r > n \}}$ satisfies 
$s_0 +n -R(s_0) < 2^{n-1}$, hence $s_0 \leq \ka(n)$ and we obtain that
$\RP^{2^n-1}$ admits no tame corank $1$ map into
$\R^{2^{n-1} + 2^{\min \{ r : r + 2^r > n \}} +j}$
for $n \geq 4$ and $j \geq 1$. Also, since 
$\min \{ r : r + 2^r > n \} \leq \lceil \log_{2}{n} \rceil$,
the same conclusion holds in the case of the target
$\R^{2^{n-1} + 2^{\lceil \log_{2}{n} \rceil} +j}$ for $n \geq 4$ and $j \geq
1$. For example, there exists neither a fold map from $\RP^{31}$ to
$\R^{21+2j}$ for $0 \leq j \leq 5$ nor a tame corank $1$ map from $\RP^{31}$
to $\R^{22+2j}$ for $0 \leq j \leq 4$.
\end{rem}

\subsection{Cobordism class of the source manifold}

For $n \equiv 0 \mod{4}$, let $X^n$ be a closed  oriented
$n$-manifold such that it is null-cobordant as an unoriented manifold and its
only nonzero Pontryagin characteristic number is $p_1^{n/4}[X^n]>0$, which 
 is equal to the minimal even value attainable by manifolds with
these properties. We define the following linear subspaces of
$\mathfrak N_n$:
\begin{itemize}
\item for $n=2^a$ with $a \geq 2$, $\mathfrak A^1$ is the
$1$-dimensional space defined by the vanishing of $w_2^{n/2}+w_n$ as well as
all monomial Stiefel-Whitney numbers except $w_2^{n/2}$ and $w_n$. For example, the 
cobordism class of $(\CP^2)^{n/4}$ generates $\mathfrak A^1$.
\item for $n=2^{b+1}+2^b-1$ with $b \geq 1$, $\mathfrak B^1$ is the
$1$-dimensional space defined by the vanishing of
\begin{itemize}
\item all monomial Stiefel-Whitney numbers not of the form $w_{m_1} \cdots w_{m_{2^b}}$,
\item all monomial Stiefel-Whitney numbers containing $w_1$,
\item all pairwise sums of the rest of monomial Stiefel-Whitney numbers.
\end{itemize}
\item for $n=2^a+2^b-1$ with $a \geq b+2$ and $b \geq 1$, $\mathfrak C^2$ is
the two-dimensional space defined by the vanishing of
\begin{itemize}
\item all monomial Stiefel-Whitney numbers which are not either of the form
$$w_{m_1} \cdots w_{m_{2^b}}  \mbox{ \ \ or \ \ }  w_{m_1} \cdots w_{m_{2^{a-1}}},$$
\item all monomial Stiefel-Whitney numbers containing $w_1$,
\item all pairwise sums of Stiefel-Whitney numbers of the form
$w_{m_1} \cdots w_{m_{2^{a-1}}}$ with all $m_j \geq 2$, and all pairwise sums of
Stiefel-Whitney numbers of the form
$w_{m_1} \cdots w_{m_{2^{b}}}$ with all $m_j \geq 2$.
\end{itemize}
\end{itemize}

\begin{thm}\label{nullcob}
Let $n \geq 2$. Assume  $M$ is an oriented $n$-manifold admitting 
a fold map into a stably parallelizable $(n-1)$-manifold.
Then either $M$ is oriented null-cobordant or one of the following cases
occurs:
\begin{enumerate}[\rm (1)]
\item
$n \equiv 0 \mod{4}$, $n$ is not a power of $2$ and $M$ is oriented cobordant
to  $mX^n$ for some $m \in \Z$.
\item
$n=2^a$ for some $a \geq 2$ and either 
\begin{enumerate}[\rm (a)]
\item
$[M]$ is the nonzero element of $\mathfrak A^1$, or
\item
$M$ is oriented cobordant to  $mX^n$ for some $m \in \Z$.
\end{enumerate}
\item 
$n = 2^{b+1} + 2^b -1$ for some positive integer $b$ and $[M] \in \mathfrak
B^1$.
\item
$n = 2^a + 2^b -1$ for some positive integers $a$ and $b$, $a \geq b+2$, and
$[M] \in \mathfrak C^2$.
\end{enumerate}
\end{thm}
 Note that in the cases (1) and (2b) $M$ is unoriented
null-cobordant. Also, the case (2a) implies $w_n[M] \neq 0$ and can
be excluded if $n \neq 2,4,8$, see \cite{SaeSak} and use the fact that a
stably parallelizable manifold is almost parallelizable.

\par
   
\begin{cor}\label{osszefogl}
If $n$ is not of the form $2^a + 2^b - 1$ for some integers $a > b \geq 0$ and the
orientable $n$-manifold $M$ has an odd Pontryagin number or a nonzero
Stiefel-Whitney number, then $M$ has no fold map into any stably
parallelizable $(n-1)$-manifold.
\end{cor} 

\begin{rem}
If $M$ is a spin manifold, then Corollary~\ref{osszefogl} holds 
with the relaxed condition   $n \neq 2,4,8$, see Corollary~\ref{spin}.
\end{rem}

\begin{rem}
By Theorem~\ref{nullcob}
if an orientable  $(4m+1)$-manifold $M$ admits a fold map into a 
stably parallelizable $4m$-manifold, then
the de Rham invariant $w_2 w_{4m-1}[M]$ vanishes.
This may suggest further relations of fold maps to surgery theory, see \cite{An0}.
\end{rem}

\begin{thm}\label{nullcoborising}
Let $n \geq 2$. There exists a $1$-dimensional linear subspace
$\mathfrak D^1 \leq \mathfrak N_n$ such that if $M$ is a possibly non-orientable $n$-manifold
 admitting a tame corank $1$ map into a stably
parallelizable $(n-1)$-manifold, then 
$[M] \in \mathfrak D^1$ and
$w_1^n[M] = 1$  if $M$ is not null-cobordant.
\end{thm}

\begin{prop}\label{nullcobnagykodimkonnyu}
Let $n>k \geq 5$ where $k$ is odd and not of the form $2^a - 1$ for some $a \geq 3$.
There exists a $1$-dimensional linear subspace $\mathfrak E^1 \leq
\mathfrak N_n$ such that if $M$ is an $n$-manifold with
$w_1(TM)=\cdots = w_k(TM) = 0$   
 admitting a fold
map into a stably parallelizable $(n-k)$-manifold, 
then $[M] \in \mathfrak E^1$. Additionally, if $[M] \in \mathfrak E^1$ and $M$ is not
null-cobordant, then $w_n[M] = 1$ is the only nonzero monomial
characteristic number of $M$.
\end{prop}

Again, \cite{SaeSak} implies that $M$ is null-cobordant if $n-k
\neq 1,3,7$.

\begin{prop}\label{nullcobnagykodimkonnyu2}
Let $n>k \geq 1$ and $k$ is odd. 
There exists a $1$-dimensional linear subspace $\mathfrak F^1 \leq
\mathfrak N_n$ such that if $M$ is an $n$-manifold which admits a
Morin map into a stably parallelizable $(n-k)$-manifold and $w_i(TM)=0$
for $i=1, \dots, k+1$, then $[M] \in \mathfrak F^1$. Additionally, if 
$[M] \in \mathfrak F^1$ and
$M$ is
not null-cobordant, then $w_n[M] = 1$ is the only nonzero monomial
characteristic number of $M$.
\end{prop}

 As before, the case of $w_n[M] \neq 0$ is excluded if $n -k \neq 5, 6$ or $n-k \geq 9$, see
\cite{Sad, SaeSak}.

\begin{rem}
By Theorem~\ref{charclass}
we can make analogous statements to the above  in the case of 
 not stably parallelizable $Q$ as well.
\end{rem}

\par

Numerical calculations similar to those of the proof of Theorem~\ref{nullcob}
suggest the following conjecture:
\begin{conj}\label{nullcobnagykodim}
Let $n>k \geq 2$ and $k = 2^a-1$, where $a \geq 2$. 
There exists a $1$-dimensional linear
subspace $\mathfrak G^1 \leq \mathfrak N_n$ such that if an
$n$-manifold $M$ with $w_i(TM)=0$ for $i=1, \ldots, k$ admits a fold map
into a stably parallelizable $(n-k)$-manifold, then we have one of the following
cases:
\begin{enumerate}[\rm (1)]
\item
$n = 2^s$ or $n=2^s +1$ with $s \geq a+1$, and $[M] \in \mathfrak G^1$.
\item
$M$ is null-cobordant.
\end{enumerate}
\end{conj}
 We verified this conjecture  for $n \leq 1200$, $3 \leq k
\leq 1023$ with the help of a computer.

\section{Perturbing the blowup of a singular map}

\begin{proof}[Proof of Theorem~\ref{perturb}]
Let $\nu$ denote the  
$1$-dimensional normal bundle of the immersion $$f|_
{\Si_{(k+1)/2}^{k+1,0}} \co \Si_{(k+1)/2}^{k+1,0} \to Q.$$ 
We identify the normal bundle $\eta$
with a tubular neighborhood of $\Si_{(k+1)/2}^{k+1,0}$ in $M$ so
 that $f$ restricted to  $\eta$ is a
 composition of 
 \begin{enumerate}[(1)]
 \item
 a
 (nonlinear) bundle map $\iota \co \eta \to \nu$  having the form
$$
 (x_1,\ldots, x_{k+1}) \mapsto \sum_{i=1}^{(k+1)/2} x_i^2- \sum_{i=(k+1)/2+1}^{k+1}
x_i^2
$$
on the unit disk of each fiber of $\eta$ in suitable local coordinates
with
\item
an immersion $\va \co \nu \to Q$.
\end{enumerate}
 Under this
identification, the points of $\bls \eta M$ in a neighborhood of
$\pi^{-1}(\Si_{(k+1)/2}^{k+1,0})$ are identified with the sets of pairs 
$[(v, \tilde v)] := \{(v,\tilde v), (v, - \tilde v)\}$ where $v$ and
$\tilde v$ are parallel vectors in the same fiber  of $\eta$ and
$\tilde v$ has length $1$.
Note that $\bls \eta f ([(v,  \tilde v)]) = \va \circ \iota(v)$
under these identifications.

\par

We define the perturbed map $\Theta \co \bls \eta M \to Q$ to agree with $\bls \eta f$
outside 
 the $\pi$-preimage of
the unit disk bundle of $\eta$ and define $\Theta$ by the  formula 
\begin{equation}\label{epsilon}
\Theta\big([(v,\tilde v)]\big)=\va \left ( \iota(v)+\ep(p) \omega(\Vert v \Vert) \iota(\tilde v) \right )
\end{equation}
within the $\pi$-preimage of
this disk bundle of $\eta$.
Here $\omega \co \R \to [0,1]$ is a bump function which is equal to $1$
around $0$ and $0$ around $1$, $p \in \Si_{(k+1)/2}^{k+1,0}$,
$v$ and $\tilde v$ are in the fiber $\eta_p$ of $\eta$ over $p$,
and $\ep(p)
 = \ep>0$ is a small real number we will choose later.
Note that $\Theta$ is well-defined since 
$\iota(\tilde v) = \iota(-\tilde v)$.

\par

Clearly the differential $d\Theta$ has rank at least $n-k-1$ outside $\pi^{-1}(\eta)$.
From (\ref{epsilon}) it is easy to see that
$d\Theta$ has rank at least $n-k-1$ on $\pi^{-1}(\eta)$ as well
since for any small curve $\al \co p \mapsto [(v, \tilde v)]_p \in \pi^{-1}(\eta_p)$
going in the fibers $\pi^{-1}(\eta_p)$ of $\pi^{-1}(\eta)$, $p\in \Si_{(k+1)/2}^{k+1,0}$,
where $v, \tilde v$ are fixed, we have that $\Th(\al(p))$ is an immersion.
Hence $\Th$ is a corank $1$ map.
To get
the singular set of $\Theta|_{\pi^{-1}(\eta)}$,
 we first take a  curve $\gamma(t) = [(t\tilde v,\tilde v)]$ in
the blowup $\pi^{-1}( \eta_p)$ of a single fiber $\eta_p$ in $\bls \eta M$, $p \in \Si_{(k+1)/2}^{k+1,0}$.
The composite map
$$
\Theta \circ \gamma (t) =\va \left  ( \iota(t\tilde v)+\ep\omega(t) \iota(\tilde v) \right ) =
\va \left ( (t^2+\ep\omega(t)) \iota(\tilde v) \right )
$$
has a single critical point at $t=0$ if  $\iota (\tilde v) \neq 0$ and $\ep$ is small enough.
Taking the curve $\de(s) = [(t_0 \tilde v_s,\tilde v_s)]$ with
a  fixed $t_0$,
$\iota(\tilde v_0) =0$
so that it intersects  $\{ [(v, \tilde v)] : \iota(\tilde v) = 0 \}$
transversally the composite map 
$$\Theta \circ \de (s) = \va ( \iota( t_0 \tilde v_s ) + \ep \om(t_0) \iota( \tilde v_s) ) =
\va \left ( (t_0^2 + \ep \om(t_0)) \iota(\tilde v_s) \right )
$$
has nonzero derivative at $s=0$.
Hence the singular points of $\Theta |_{\pi^{-1}( \eta_p)}$
are contained in $\pi^{-1}(p)$ and
the singular points of $\Th |_{\pi^{-1}( \eta)}$ are contained in $\pi^{-1}(\Si_{(k+1)/2}^{k+1,0})$.

Clearly a critical point of $\Theta |_{\pi^{-1}( \eta_p)}$ is 
a critical point of $\Theta |_{\pi^{-1}( p)}$.
In the following, we show that
at the critical points of 
 $\Theta |_{\pi^{-1}( p)}$
the composite  map $\Theta \circ \gamma (t)$, where
$\ga(t)= [(t\tilde v,\tilde v)]$, has a critical point  for $t=0$.
Hence any critical point of 
$\Theta |_{\pi^{-1}( p)}$ is a critical point of $\Theta |_{\pi^{-1}( \eta_p)}$.
The choice of coordinates $x_1,\dots,x_{k+1}$
on $\eta_p$ identifies $\pi^{-1}(p)$ with the projective space $\RP^{k}$,
and the restriction $\Theta|_{\pi^{-1}(p)}$ is equal to
$$
\Th|_{\pi^{-1}(p)} \co [x_1:\dots:x_{k+1}]  \mapsto
  \va \left( \ep \iota \left( \frac{(x_1,\dots,x_{k+1})}{\Vert
(x_1,\dots,x_{k+1}) \Vert} \right) \right).
$$

This map is Morse-Bott and has critical points along two copies of
$\RP^{(k+1)/2-1}$, which are  $\{ [x_1:\cdots:x_{(k+1)/2}:0:\cdots:0] \in
\RP^{k} \}$ and  $\{
[0:\cdots:0:x_{(k+1)/2+1}:\cdots:x_{k+1}] \in \RP^{k} \}$, and it is easy to see that 
both critical loci have
index $(0,(k+1)/2-1)$.
Therefore 
$\Th|_{\pi^{-1}(\eta_p)}$ is a Morse-Bott map with indices $(1,(k+1)/2-1)$ and
with this two copies of $\RP^{(k+1)/2-1}$
as singular set. Hence $\Th$ is a Morse-Bott map with the corresponding indices.
We also have that the singular set of $\Th|_{\pi^{-1}(\eta)}$
is a fiber bundle with  the singular sets of $\Th|_{\pi^{-1}(\eta_p)}$,
$p \in \Si_{(k+1)/2}^{k+1,0}$, as fibers.
It is easy to see that 
the $(n-k-1)$-manifold
$q_{\sim} (\Si_{(1,(k+1)/2-1)} \cap \pi^{-1}(\eta))$
has an embedding 
into the sphere bundle of $\nu$ given by the perturbation.
Furthermore, $q_{\sim} (\Si_{(1,(k+1)/2-1)} \cap \pi^{-1}(\eta))$
is a double covering of $\Si_{(k+1)/2}^{k+1,0}$ given by this embedding,
and it  is immersed 
with a trivial normal bundle $\tilde \nu$
into
the tubular neighborhood of 
$f ({\Si_{(k+1)/2}^{k+1,0}})$.
Moreover there is a natural trivialization of $\tilde \nu$
corresponding to the indices of the singular set of 
$\Th|_{\pi^{-1}(p)}$,
$p \in \Si_{(k+1)/2}^{k+1,0}$.
Thus  the perturbation $\Th$ satisfies the requirements of the theorem.
Hence $\Th$ is a tame corank $1$ map and applying Proposition~\ref{Kgamma} 
finishes the proof.
\end{proof}

\begin{rem}
The  double covering of $\Si_{(k+1)/2}^{k+1,0}$ defined by 
$\Th(\Si_{(1,(k+1)/2-1)} \cap \pi^{-1}(\eta))$
 is trivial if and only if
$\nu$ is trivial. 
\end{rem}

\begin{rem}
If $k=1$, then $\Th$ is obviously a fold map.
For example, for a Morse function $f \co S^2 \to \R$ with three definite and
one indefinite critical points, $\bls \eta S^2 = S^2 \# \RP^2  = \RP^2$ and
$\Th$ is a Morse function with three definite and two indefinite critical
points. It can be seen that $\Th$ has two singular fibers containing
indefinite critical points and exactly one of them has non-orientable
neighborhood.
\end{rem}

\begin{lem}\label{thm:LS}
Let $X$ be a manifold and $l$ be a line bundle over $X$. Assume that there
is an open covering $X_0 \cup X_1 = X$ such that the bundle $l$ is trivial
over both $X_0$ and $X_1$. Then there exists an epimorphism $\varepsilon^2_X \to
l$.
\end{lem}
\begin{proof}
Let $f_i \co l|_{X_i} \to \R$ be fiberwise linear isomorphisms which
trivialize $l$ over $X_i$ for $i=0,1$. Since $X_0$ and $X_1$ are open, we
can choose continuous functions $\lambda_0, \lambda_1 \co X \to [0,1]$ such
that $\lambda_i^{-1}(0) = X - X_i$ for $i=0,1$. Define the fiberwise linear
function $f \co l \to \R^2$ by the formula
$$
f(v) = (\lambda_0(x) f_0(v), \lambda_1(x) f_1(v))
$$
for each vector $v \in l$ over the point $x \in X$. This definition makes
sense since $\lambda_i(x)=0$ whenever $f_i(x)$ is not defined. Over $X_i$,
the map $f$ composed with the projection onto the $i$-th coordinate is a
nonzero rescaling of $f_i$, $i = 0, 1$, thus $f$ is injective and we can identify $l$ with
a subbundle of $\varepsilon^2_X$. Taking a Riemannian metric on
$\varepsilon^2_X$ allows us to consider the orthogonal projection
$\varepsilon^2_X \to l$, which is an epimorphism.
\end{proof}

\begin{proof}[Proof of Theorem~\ref{perkov}]
Suppose $f \co M^n \to Q^{n-k}$ is a Morin map.
By \cite{Che} the cokernel bundle of $df \co
TM \to f^*TQ$ is trivial over a neighborhood of $\Si^{k+1,1}$ in $\Si$.
Note that  $\coker df$ has  natural orientations over the 
 fold singularities of index not equal to $(k+1)/2$,
and these orientations agree when two fold singular sets are attached 
to each other
along $\Si^{k+1,1}$. 
Let $\Si_{\geq 0}$ denote 
the complement of a small regular neighborhood of $\Si^{k+1,1} \cup \bigcup_{\la \neq
(k+1)/2} \Si^{k+1,0}_\la$ in $\Si$. Then
$\coker df$ is trivial on $\Si - \Si_{\geq 0}$.
If $\Si_{\geq 0} = \emptyset$, then $f$ is a tame corank $1$ map and
by Proposition~\ref{Kgamma} we obtain the statement of the theorem.
Assume that $\Si_{\geq 0} \neq \emptyset$.

\par

Take the blowup $\bls \csi f \co \bls \csi M \to Q$. 
Define a smooth non-negative function
$\Delta$ on ${\Si_{\geq 0}}$ such that 
it 
vanishes on
${\Si_{= 0}}$, where $\Si_{= 0}$ is a
 small compact regular neighborhood  of $\del \Si_{\geq 0}$, and takes  small positive values on
 ${\Si_{> 0}}$, where 
  $\Si_{>0} = \Si_{\geq 0} - \Si_{=0}$.
Then, perturb the map $\bls \csi f \co \bls \csi M \to Q$
in $\pi^{-1}(\eta|_{\Si_{\geq 0}})$ in the same way as we did in Theorem~\ref{perturb}
but replace $\ep$ in \eqref{epsilon} by the value of the function $\Delta$.
Denote the resulting map by $\widetilde {\bls \csi f}$.
Note that 
on $\bls \csi M - \pi^{-1}(\eta|_{\Si-\Si_{>0}})$ the restrictions of
$\widetilde {\bls \csi f}$ and $\bls \csi f$
coincide.
Denote by $\tilde \Si_{>0}$ the singular set of
 $\widetilde {\bls \csi f}|_{\pi^{-1}(\eta|_{\Si_{>0}})}$.
The singular set of $\widetilde {\bls \csi f}|_{\pi^{-1}(\csi|_{\Si - \Si_{>0}})}$
is equal to $\pi^{-1}(\Si - \Si_{>0})$.

Since $d \bls \csi f (T\bls \csi M) =df d\pi (T\bls \csi M) =  df(TM)$ and
$f^*TQ / f^* df(TM)$ is trivial on $\Si - \Si_{>0}$, clearly
$\coker d {\bls \csi f} = \pi^*f^*(TQ) / \pi^* f^* d\bls \csi f(T\bls \csi M)$ is also trivial 
on $\pi^{-1}(\Si - \Si_{>0})$.
Hence so is $\coker d \widetilde{\bls \csi f}$. Moreover
  $d \widetilde {\bls \csi f}$
 also has trivial cokernel on  $\tilde \Si_{>0}$
  by 
 Theorem~\ref{perturb}.

Apply Lemma \ref{thm:LS} to $l = \coker d\widetilde {\bls \csi f}$ over 
$X = \tilde \Si_{>0} \cup \pi^{-1}(\Si - \Si_{>0})$
with the covering consisting of $X_0 = \tilde \Si_{>0}$ and $X_1$ 
being a small neighborhood of $\pi^{-1}(\Si - \Si_{>0})$. 
The argument above ensures that $l$ is indeed trivial when
restricted to either $X_0$ or $X_1$
since $X_1$ is a deformation retract of $\pi^{-1}(\Si - \Si_{>0})$. Therefore there exists a fiberwise
epimorphism $\sigma : \varepsilon^2_X \to \coker d\widetilde {\bls \csi f} |_X$. 
Compose $\sigma$ with the standard embedding
$\coker d\widetilde {\bls \csi f}|_X \to (\widetilde {\bls \csi f})^*TQ|_X$ and then
extend this composite map 
to all of ${\bls \csi M}$ as a linear bundle map $\tilde\sigma :
\varepsilon^2_{ {\bls \csi M}} \to (\widetilde {\bls \csi f})^*TQ$
by scaling it with a bump function
concentrated on a small neighborhood of $X$. Combining $d\widetilde {\bls \csi f}$ with
$\tilde\sigma$ we get a bundle map
$$
d\widetilde {\bls \csi f}+\tilde\sigma : T{\bls \csi M} \oplus \varepsilon^2 \to (\widetilde {\bls \csi f})^*TQ
$$
which is obviously surjective both on $\bls \csi M - X$ and $X$.
This completes the proof.
\end{proof}

\section{Computing the characteristic classes of the source manifold}

Let $\ga$ denote the line bundle over $\Si$ defined by the condition that
$w_1(\ga)$ is  Poincar\'e dual to the class represented by
$\Si^{k+1,1}$. We relate $f^* df (TM)|_\Si$ to $T\Si$ by the following

\begin{prop} 
For a cusp map $f$, we have
$T\Si  \+ \ga \cong f^*df(TM)|_\Si \+ \ep^1$.
\end{prop}

\begin{proof}
Denote by $C$ the manifold $\Sigma^{k+1,1}$.
Since $f$ is a cusp map, we have $C=\Sigma^{k+1,1,0}$.
We will first construct a bundle monomorphism
$$
i \co T\Si \to f^*df(TM)|_\Si \+ \ep^1
$$
covering the identity map of $\Si$. Apart from $C$, the
map $df$ is an isomorphism between $T\Si$ and $f^*df(TM)|_\Si$. On
$C$ the restriction of $df$ to
$TC$ is a monomorphism hence there is an isomorphism
$$
j \co T\Sigma|_{C} \to f^*df(T{C})|_{C} \+ \ker d(f|_{\Si})|_{C}
$$
defined by  composing  the isomorphism
$T\Sigma|_{C} \cong TC \+ \ker d(f|_{\Si})|_{C}$ with
the isomorphism
$$
(df|_{TC} , {\mathrm {id}}_{\ker d(f|_{\Si})|_{C}}) \co 
TC \+ \ker d(f|_{\Si})|_{C} \to 
f^*df(T{C})|_{C} \+ \ker d(f|_{\Si})|_{C}.
$$
Denote by ${\mathrm {pr}}_{\ker} \co
T\Si|_{C} \to \ker d(f|_{\Si})|_{C}$
the composition of $j$ with the projection  of
$f^*df(T{C}) \+ \ker d(f|_{\Si})|_{C}$
to the second factor.

\par

Since $\ker d(f|_{\Si})|_{C} \subset T\Si$ is never tangent to $C$, we can
identify it with the normal bundle of $C$ in $\Si$. But $k$
is odd hence this normal bundle is trivial -- the indices of fold
points on the two sides are different. After choosing a
trivialization of $\ker d(f|_{\Si})|_{C}$, the map ${j}$ can be considered as an
embedding of $T\Si|_{C}$ into $f^*df(TM)|_{C} \+ \ep^1$, with its
image $\im {j} = f^*df(T{C})|_{C} \+ \ep^1$. This embedding extends
as a fiberwise embedding onto a small neighborhood $N(C)$ 
of $C$ in
$\Si$, and we will consider ${j}$ 
and also ${\mathrm {pr}}_{\ker}$ to be defined on $N(C)$.

\par

Define $i$ to be the linear interpolation of 
the embedding $j \co T\Si|_{N(C)} \to f^*df(TM)|_{N(C)} \+ \ep^1$
and the map $(df |_{T\Si},0) \co T\Si
\to f^*df(TM)|_\Si \+ \ep^1$. That is, we take a bump function $\lambda \co T\Si
\to [0,1]$ such that $\lambda=0$ outside over a tubular neighborhood of
$C$ in $\Si$ and $\lambda^{-1}(\{1\})=T\Si|_{C}$, and we
define $i$ to be $(df |_{T\Si},\lambda {\mathrm {pr}}_{\ker}) \co T\Si \to 
 f^*df(TM)|_{\Si} \+ \ep^1$. Thus $i$ is well-defined, since
$\lambda=0$ where ${\mathrm {pr}}_{\ker}$ is not defined, and it is clear that $i$ has
full rank both on $C$ and its complement in $\Si$.

\par

From this embedding $i$, we get $f^*df(TM)|_\Si \+ \ep^1 \cong T\Si \+ \coker
i$, and we only need to identify $\coker i$ with $\gamma$. Indeed, on the
set $\Si - C$ the line bundle $\coker i$ is trivial as
$\im i$ projects isomorphically onto $f^*df(TM)|_\Si$. On a tubular neighborhood
of  $C$,  this trivialization of $\coker i|_{\Sigma - C}$
has the opposite signs on the two sides of $C$, thus $w_1(\coker
i)$ is dual to $C$ in $\Si$ as claimed.
\end{proof}

\begin{cor}\label{thm:tsigma}
For a cusp map $f$, we have $w(T\Si)=w({f|_\Si}^*TQ)w(\nu)^{-1}w(\ga)^{-1}$,
where $\nu$ denotes the line bundle $f^*TQ|_\Si / f^*df(TM)|_\Si$.
\end{cor}

Let $c \in H^1(\Sigma)$ denote the characteristic class $w_1(\ga)$. As noted
above, the manifold $\Si^{k+1,1}$ has a trivial normal bundle in $\Si$,
thus $c^2=0$. Let $b$ denote $w_1( f^*TQ|_\Si / f^*df(TM)|_\Si ) \in H^1(\Sigma)$.

\begin{proof}[Proof of Theorem~\ref{charclass}]
 Let $\de$ be $0$ if $f$ is a fold map, and let $\de$ be $1$ otherwise.
If 
 the Morin map $f \co M^n \to Q^{n-k}$ 
 is not a cusp map and both $M$ and $Q$
are orientable, then perturb $f$ to get a cusp map, see \cite{Sad}, and 
denote this cusp map by $f$ as well for simplicity.

By the blowup formula for Stiefel-Whitney classes 
\cite[Theorem~10 and Remark~(2) on page 328]{GePa},
we can express in our notation the Stiefel-Whitney classes of $TM$ in terms of the classes
of $T\bls \csi M$, $T\Si$ and $\csi$ in the following way:
\begin{multline*}
w(T\bls \csi M)-w(\pi^*TM) = \\ =
{\tilde \imath}_{!} \left (w(p^*T\Si) \frac{1}{w_1(\mu)} \left (\sum_{t=0}^{k+1}w_t(p^*\csi)(1+w_1(\mu))^{k+2-t} - w(p^*\csi)\right )
\right).
\end{multline*}
Here $\mu$ denotes the canonical line bundle 
over $\pi^{-1}(\Si)$.
Recall that $p$ is the restriction $\pi|_{\pi^{-1}(\Si)}$, and the map
${\tilde \imath} \co \pi^{-1}(\Si) \to \bls \csi M$ is
 the natural embedding.  
 By
Theorem~\ref{perkov} and
Remark~\ref{foldpertkov}, the total Stiefel-Whitney class of $T\bls \csi M$ is equal to 
the product of the total Stiefel-Whitney 
class of a $(k+1+\de)$-dimensional bundle and the
total Stiefel-Whitney class of $\pi^* f^* TQ$, which
contains no term of degree greater than $K$. Therefore $w_l(T\bls \csi M) = 0$ 
for $l > k+K+1+\de$. 

\par

Expanding the blowup formula for $r \geq k+K+2+\de$ we thus get
\begin{multline*}
w_r(\pi^* TM) = w_r(T \bls \xi M) + \\ 
+{\tilde \imath}_{!} \left (\sum_{q=0}^{r-1}w_{r-1-q}(p^*T\Si) \left [
\sum_{t=0}^{k+1} w_t(p^*\xi) \sum_{s=0}^{k+1-t}
\binom{k+2-t}{s+1} w_1(\mu)^s \right ]_{deg=q} \right ) = \\
={\tilde \imath}_{!} \left (\sum_{q=0}^{k+1} p^* w_{r-1-q}(T\Si)
\sum_{t=0}^{k+1} \binom{k+2-t}{q-t+1} p^*w_t(\xi) w_1(\mu)^{q-t} \right ),
\end{multline*}
where we use our convention about binomial coefficients and
$w_1(\mu)^{-1}$ is defined to be $0$.

The classes $w_{r-1-q}(T\Si)$ can be obtained from
Corollary~\ref{thm:tsigma}. Under the assumption that $m \geq K + \de$, we have
\begin{multline*}
w_m (T\Si) = \left[ w({f|_\Si}^* TQ) w( \nu)^{-1} w(\gamma)^{-1}
\right]_{\operatorname{deg}=m} =\\= \sum_{l=0}^{m} {f|_\Si}^*w_l(TQ) (b^{m-l}
+ b^{m-l-1}c)=\sum_{l=0}^{K} {f|_\Si}^*w_l(TQ) (b^{m-l} + b^{m-l-1}c).
\end{multline*}
Notice that in the formula all the exponents of $b$ are at least $0$.
Substituting 
  $m=r-1-q$, 
where $0 \leq q \leq k+1$,
we get that
$$
w_{r-1-q}(T\Si) = b^{r-1-q-K-\de} w_{K+\de}(T\Si)
$$
for all $0 \leq q \leq k+1$. Hence we have that if $r
\geq k+K+2+\de$, then
\begin{multline*}
w_r(\pi^* TM) =\\
= {\tilde \imath}_{!} \left (\sum_{q=0}^{k+1} p^* \left(
b^{r-1-q-K-\de}w_{K+\de}(T\Si)\right) \sum_{t=0}^{k+1} \binom{k+2-t}{q-t+1}
p^*w_t(\xi) w_1(\mu)^{q-t} \right) =\\
= {\tilde \imath}_{!} \left ( p^*b^{r-k-K-2-\de} \sum_{q=0}^{k+1}
\sum_{t=0}^{k+1} \binom{k+2-t}{q-t+1} p^* \left( b^{k+1-q} w_{K+\de}(T\Si)
w_t(\xi) \right) w_1(\mu)^{q-t} \right).
\end{multline*}
Notice that the double sum in this formula does not depend on $r$ at all, and let
$\al$ denote 
$$
 \sum_{q=0}^{k+1}
\sum_{t=0}^{k+1} \binom{k+2-t}{q-t+1} p^* \left( b^{k+1-q} w_{K+\de}(T\Si)
w_t(\xi) \right) w_1(\mu)^{q-t}.
$$
Then
$$
w_r(\pi^* TM) = {\tilde \imath}_! \left( \alpha p^*b^{r-k-K-2-\de} \right)
$$
holds for all $r \geq k+K+2+\de$ and we can calculate
products of these characteristic classes by repeatedly applying the formula
$\tilde \imath_!(u) \tilde \imath_!(u) = \tilde \imath_!(1) \tilde
\imath_!(uv)$ as follows.
For $r_1, \ldots, r_m \geq k+K+2+\de$, we have 
\begin{multline*}
w_{r_1}(\pi^* TM) \dots w_{r_m}(\pi^* TM) = \prod_{i=1}^{m} {\tilde\imath}_!
\left( \alpha p^*b^{r_i-k-K-2-\de} \right) = \\
= {\tilde\imath}_!(1)^{m-1}
{\tilde\imath}_! \left( \prod_{i=1}^{m} \alpha p^*b^{r_i-k-K-2-\de}\right) =
{\tilde\imath}_!(1)^{m-1} {\tilde\imath}_! \left( \alpha^m
p^*b^{\sum_{i=1}^{m} \left( r_i -k-K-2-\de \right)} \right).
\end{multline*}
This expression clearly depends only on $m$ and the sum $r_1+\dots+r_m$, and
since the homomorphism $\pi^* \co H^*( M; \Z_2) \to H^*(\bls \csi M; \Z_2)$ is injective, this proves the statement of the theorem.
\end{proof}

\begin{lem}\label{thm:middleNull}
Let $n=2^D+m$, $0 \leq m<2^D$. Then $\binom{n}{m}$ is odd, and $\binom{n}{r}$ is
even for all $r$ satisfying $m<r<2^D$.
\end{lem}
\begin{proof}
A criterion of \cite{Gla} states that $\binom{b}{a}$, $0 \leq a \leq b$, is
even if and only if there is a binary position at which $a$ has the digit
$1$ and $b$ has the digit $0$. This criterion shows that $\binom{n}{m}$ is
odd. If $\binom{n}{r}$ is odd for some $0 \leq r < 2^D$, then all the binary
digits of $r$ at the positions where $n$ has $0$ have to be $0$ as well.
Since $r$ has binary length at most $D$, this is equivalent to the condition
that $r$ has binary digits $1$ only at positions where $m = n-2^D$ has $1$
as well, hence the maximal such $r$ is $m$ as claimed.
\end{proof}
\begin{prop}\label{RPfold}
Let $n=2^D+m$ with $0\leq m < 2^D-2$. Assume  there exists an integer
$l$ such that the equations $w_aw_b(\R P^n)=w_cw_d(\R P^n)$ hold for all
$a,b,c,d \geq l$. Then $l \geq m+1$.
\end{prop}
\begin{proof}
Denote the generator of $H^1(\R P^n;\Z_2)$ by $x$, then we have
$$
H^*(\R P^n;\Z_2) = \Z_2[x]/(x^{n+1})
$$
and
$$
w_j(\R P^n) = \binom{n+1}{j} x^j.
$$
By Lemma \ref{thm:middleNull} 
$\binom{n+1}{m+1} = \binom{n}{m}+ \binom{n}{m+1}$ is odd,
hence the class $w_{m+1}(\R P^n)$ is the generator $x^{m+1}$, while the
classes $w_{m+2}(\R P^n), \ldots,w_{2^D-1}(\R P^n)$ vanish. Note that there
is at least one class in this latter list due to the constraint $m<2^D-2$.
In particular, the class $w_mw_{m+2}(\R P^n)$ also has to vanish, while
$w_{m+1}^2(\R P^n) = x^{2m+2}$ is not zero as $2m+2 < 2^D-2+m+2 = n < n+1$.
Therefore the relation $w_{m+1}^2=w_{m+2}w_m$ does not hold on $\R P^n$,
implying  $l>m$.
\end{proof}
\begin{prop}
Let $n=2^D+m$ with $0\leq m < 2^D-2$. Then the relations
$$\prod_{i\in I}w_i(\R P^n)=\prod_{j\in J}w_j(\R P^n)$$
hold for all $I,J \subseteq \{ 0, \dots, n \}$ which satisfy $\vert 
I \vert = \vert J \vert$, 
$\min I, \min J \geq m+1$ and
$\sum_{i\in I} i = \sum_{j \in J} j$.
\end{prop}
\begin{proof}
As before, we note that the classes $w_{m+2}(\R P^n), \ldots, w_{2^D-1}(\R P^n)$
vanish.

\par

For $\vert I \vert =1$ the statement  is trivial. For $\vert I
\vert \geq 2$ such that $\min I \geq m+1$ we have three possibilities:
\begin{itemize}
\item $I$ consists of a number of copies of $m+1$. Then the only $J$ which
satisfies both $\min J \geq m+1$ and $\sum_{j \in J} j = \sum_{i \in I} i =
(m+1) \vert I \vert$ is $I$ itself.
\item $I$ contains an index between $m+2$ and $2^D-1$. Then $\prod_{i\in I}
w_i(\R P^n)$ contains a zero class and thus vanishes.
\item $I$ contains at least one index greater than  $2^D-1$. Then taking any
other index $j \in I$ we have $\sum_{i\in I} i \geq 2^D+j \geq 2^D+m+1 =
n+1$. Therefore $\prod_{i\in I} w_i(\R P^n)$ has degree greater than $n$ and
consequently vanishes.
\end{itemize}
Observe that for any $J$ satisfying the requirements of the proposition we
have the analogous three possibilities, hence any such $J$ gives 
the same product of Stiefel-Whitney classes as $I$.
\end{proof}
\begin{rem}
In the  cases $n=2^D-2$ and $n=2^D-1$ the nontrivial
characteristic classes of $\R P^n$ are either all the generators of the
respective cohomology groups $H^*(\R P^n)$ or all vanish, therefore our
multiplicativity condition is satisfied for all indices.
\end{rem}

\begin{proof}[Proof of Theorem~\ref{szingtusk}]
Equip $TM$ and $TQ$ with Riemannian metrics, thus identifying sections of
these bundles with $1$-forms. Assume that we have a trivialization of $TQ \+
\ep^l$ given by a collection of $n-k+l$ linearly independent $1$-forms. Then
any smooth map $f \co M \to Q$ defines pullbacks of these $1$-forms to $TM
\+ \ep^l$ via $df$. By the assumption of the theorem, $\rank df \geq n-k-1$
at all points of $M$, thus the linear span of the pulled-back forms is
at least $n-k+l-1$. The metric on $TM$ identifies these forms with $n-k+l$
vector fields which have a linear span of dimension at least $n-k+l-1$
everywhere. By \cite{Gab, Po, Ro}, the rational Pontryagin class
$p_i^{\Q}(TM)=p_i^{\Q}(TM \+ \ep^l)$ is represented by the locus where $n+l-2i+2$
generic sections of $TM \+ \ep^l$ lie in a subspace of dimension at most
$n+l-2i$. This class therefore vanishes if $n-k+l \geq n+l-2i+2$, that is,
when $2i \geq k+2$.
\end{proof}


\begin{proof}[Proof of Proposition~\ref{Kgamma}]
(2) $\Longrightarrow$ (1):
By \cite{An}, if there is a $TM \+ \ep^1 \to TQ$ epimorphism, then
there is 
a fold map ${M} \to Q$ with orientable singular set.
(1) $\Longrightarrow$ (2):
Assume that we have a tame corank $1$ map $f \co M \to Q$.
The  bundle $\coker df |_\Si = (f^*TQ / f^*df(TM))|_\Si$  is considered as 
a subbundle of $f^* TQ$ and it is trivial.
Similarly to \cite[Proof of Lemma~3.1]{An},
let $L \co \ep^1 \to TQ$ be an extension of the bundle monomorphism $\coker df |_\Si
\to f^*TQ \to TQ$ as a bundle homomorphism covering $f$. Then $df + L$ is an
epimorphism $TM \+ \ep^1 \to TQ$.

Finally, if (1) or (2) holds and $Q$ is stably parallelizable, then
by the above,  we have 
$TM \+ \ep^1 \+ \ep^N \cong \zeta \+ f^*TQ \+ \ep^N \cong \zeta \+ \ep^{N+n-k}$ for some
$N \gg 0$ and a $(k+1)$-dimensional bundle $\zeta$. Thus $g.dim ([TM] - [\ep^n]) \leq k+1$.

If $Q$ is stably parallelizable and
$g.dim ([TM] - [\ep^n]) \leq k+1$, then 
 $TM \+ \ep^N \cong \zeta^{k+1} \+ \ep^{N + n-k-1} \cong \zeta^{k+1} \+ TQ \+ \ep^{N-1}$
  for some $N \gg 0$, and thus 
$TM \+ \ep^1 \cong \zeta^{k+1} \+ TQ$, which proves (2).
\end{proof}

\begin{rem}
If there is a tame Morse-Bott map $f \co P^{n} \to Q^{n-k}$, then 
$TP \oplus \ep^1$ splits as
$\zeta^{k+1} \oplus  f^*TQ$  for some
$(k+1)$-dimensional vector bundle $\zeta^{k+1}$.
\end{rem}

\begin{proof}[Proof of Proposition~\ref{alkgamma}]
Let $\va(n)$ denote the cardinality of the set $\{  0 < s \leq n : s \equiv 0,1,2,4 \mod{8} \}$.
By \cite[\S5]{At}, $[T{\RP^n}] - [\ep^n] = (n+1)x$ and $\ga^i([T{\RP^n}] - [\ep^n])
 = 2^{i-1} \binom{n+1}{i}x$, $i \geq 1$,
where $x$ denotes the generator of $\rKO(\RP^n) = \Z_{2^{\va(n)}}$.
Therefore
$\ga^i([T{\RP^n}] - [\ep^n]) = 0$ if and only if $2^{\va(n)}$ divides $2^{i-1} \binom{n+1}{i}$.
Let $r(n)$ denote the greatest integer $s$ for which $2^{s-1}
\binom{n+1}{s}$ is not divisible by $2^{\va(n)}$.
Then by Proposition~\ref{Kgamma} there is no tame corank $1$ map of
$\RP^{2^n-1}$ into $\R^{2^n-1-k}$ for $k \leq {r(2^n-1)-2}$.
It is easy to see that $\va(2^n-1) = 2^{n-1}-1$ if $n \geq 3$.
By a classical result of E.\ Kummer,
the highest power $c(s)$ of $2$ which divides $\binom{2^n}{s}$ 
can be obtained by counting the number of carries 
when $s$ and $2^n -s$ are added in base $2$. 
For $s \leq 2^{n-1}-1$, we claim that $c(s) = n - R(s)$, where 
$2^{R(s)}$ is the maximal power of $2$ which divides $s$. Indeed,
$2^n-1-s$ is obtained by negating the binary form of $s$ bitwise,
hence $2^n -s$ is obtained by negating the binary form of $s$ bitwise from
the $(n-1)$st to the $R(s)$th binary position, where both of $s$ and $2^n-s$
have the digit $1$, and after that position both have digits $0$.
Therefore when we add $s$ and $2^n-s$ in base $2$, we have $n-R(s)$ carries. 
By the definition of $r(n)$ it follows that $r(2^n-1)$ is the largest integer $s$ for which
$s+n -R(s) < 2^{n-1}$. 
\end{proof}

\subsection{Computing the cobordism class of the source manifold}

Theorem~\ref{charclass} gives us 
relations among the characteristic {\it numbers} of a source manifold of a Morin map as well.
However, by following a different line of argument, we can obtain more relations among
the characteristic numbers as follows.

For a Morin map $f \co M^{n} \to Q^{n-k}$ with odd $k \geq 1$,
let us denote by ${N_{\Si}}$ the projectivization $\RP(\csi \oplus \varepsilon^1)$
of the $(k+2)$-dimensional vector bundle
$\csi \oplus \ep^1$ over the singular set $\Si$,
where  $\csi$ denotes the normal bundle of $\Si$.
 Thus  ${N_{\Si}}$ is 
a closed $n$-dimensional manifold fibered over $\Si$ with 
$\RP^{k+1}$ as fiber. Let $\tau \co {N_{\Si}} \to \Si$ denote this fibration.
\begin{lem}\label{blowupkobord}
The blowup $\bls \csi M$ is  cobordant to the disjoint union of
$M$ and ${N_{\Si}}$.
\end{lem}
\begin{proof}
Consider the disk bundle $D(\csi \oplus \ep^1)$ of $\csi \oplus \varepsilon^1$. Let $U$ and $V$
 be small neighborhoods of $\csi \+ \{ 1 \}$ and $\csi \+ \{ -1 \}$
respectively in the boundary $\partial D(\csi \oplus \ep^1)$. 
The total space of $D(\csi \oplus \ep^1)$ can be naturally 
glued to the boundary component 
$\left( M \sqcup \R P(\csi \+ \varepsilon^1)
\right) \times \{0\}$ of 
 $\left( M \sqcup \R P(\csi \+ \varepsilon^1)
\right) \times [0,1]$ along $U$ identified with the total space of $\xi$ as
an open submanifold of $M$ and along $V$ identified with $[\csi:1] \subset
\RP(\csi \+ \varepsilon^1)$. After smoothing the corners introduced by the
gluing, the resulting $(n+1)$-manifold has boundary consisting of the
disjoint union of $M$, $\RP(\csi \+ \varepsilon^1)$ and $\bls \csi M$.
This completes the proof.
\end{proof}
\begin{rem}\label{mapkobord}
As one can see easily, the cobordism in Lemma~\ref{blowupkobord} extends
naturally to a bordism of
the maps $\bls \xi f \co \bls \csi M \to Q$ and the union  $f \sqcup f|_\Si \circ \tau 
\co M \sqcup N_\Si \to Q$. Indeed, we
can map all points of each fiber  of $D(\csi \oplus \ep^1)$ over $p$ to $p$, where
$p \in \Si$, and then into $Q$ by $f|_\Si$. Thus the
evaluation of any ``characteristic number'' $w_I$ (i.e.\ degree $n$ monomial of Stiefel-Whitney
characteristic classes) of $T\bls \xi M-(\bls \xi f)^*TQ$ on the fundamental
class  $[\bls \xi M] \in H_n(\bls \xi M; \Z_2)$ is equal to the sum of the evaluations of 
  $w_I(TM - f^*TQ)$ and $w_I(TN_\Si - \tau^*  f|_\Si^*TQ)$
 on the fundamental classes $[M]$ and $[N_\Si]$, respectively.
\end{rem}

Recall that $b$ denotes $w_1( f^*TQ|_\Si / f^*df(TM)|_\Si) \in H^1(\Sigma)$ and
 $w_1(\ga) = c \in H^1(\Sigma)$ is the Poincar\'e dual to the class represented by
$\Si^{k+1,1}$. We have also seen that $c^2 = 0$.

\begin{prop}\label{karosztalyok}
Let $f \co M^n \to Q^{n-k}$ be a cusp map.
Let $\de$ be $0$ if $f$ is a fold map, and let $\de$ be $1$ otherwise.
\begin{enumerate}[\rm (1)]
\item
For $r \geq k+1+\de$, the degree $r$ term of 
$w(T{N_{\Si}} - \tau^*{f|_\Si}^*TQ)$  
has the form $$\tau^*b^{r-k-1-\de} w_{k+1+\de}(T{N_{\Si}}-\tau^*{f|_\Si}^*TQ).$$
\item
Any two characteristic numbers of $w(T{N_{\Si}}-\tau^*{f|_\Si}^*TQ)$ 
 which contain the same number of multiplicands and contain no
instances of $w_1, \ldots, w_{k+\de}$ are equal.
\item
For any multiindex  $J = (j_1, \dots, j_l)$ such that $\sum_{i=0}^{l} j_i
= n$ and $j_i \geq k+2+\de$ for some $1\leq i \leq l$, the
characteristic numbers $\langle w_J(TM-f^*TQ), [M] \rangle$ and
$\langle w_J(T{N_{\Si}}-\tau^*{f|_\Si}^*TQ), [N_\Si] \rangle$ coincide.
The characteristic numbers defined by $w(TM-f^*TQ)$ which involve  
no $w_1, \ldots, w_{k+\de}$ satisfy the property of depending only on the
number of multiplicands.
\end{enumerate}
\end{prop}

\begin{proof}
 The fibration ${\tau} \co {N_{\Si}} {\to}
\Sigma$ has fiber $\RP^{k+1}$ and $T{N_{\Si}}$ splits into
the direct sum of the horizontal component ${\tau}^* T\Sigma$ and the vertical
component $\psi$ having rank
$k+1$, which is tangent to the fibers.
By Corollary~\ref{thm:tsigma}, we have
$$
w({f|_\Si}^*TQ)^{-1}w(T\Sigma) = w(\nu)^{-1} w(\ga)^{-1} = (1+b)^{-1} (1+c)^{-1}
= \sum_{i=0}^{\infty} b^i (1 + c).
$$
Hence we can express $w(\tau^*{f|_\Si}^*TQ)^{-1}w(TN_{\Sigma})$ as
\begin{multline*}
w(\tau^*{f|_\Si}^*TQ)^{-1}w(TN_\Si) = w(\tau^*{f|_\Si}^*TQ)^{-1}{\tau}^*w(T\Sigma)w(\psi) =\\
= \left( 1+ \sum_{i=1}^{\infty} (\tau^*b^i + \tau^*b^{i-1}c) \right)
\left( \sum_{j=0}^{k+1} w_j(\psi) \right)  = \\ \overset{(a)}{=}
 \sum_{{ r=0}}^{\infty} 
\sum_{j=0}^{\operatorname{min} \{ r, k+1\} }  w_j(\psi)
(\tau^*b^{r-j} + \tau^*b^{r-j-1}c),
\end{multline*}
where (a) follows from rearranging
the sums by $i+j = r$, and
we use the convention that $b^{-1} = 0$. The class
$w_r(TN_\Sigma-\tau^*{f|_\Si}^*TQ)$ in the case of $r \geq k+2$ therefore has the
form
\begin{multline*}
w_r(TN_\Sigma-\tau^*{f|_\Si}^*TQ) = 
\sum_{j=0}^{k+1} w_j(\psi) (\tau^*b^{r-j} +\tau^*b^{r-j-1}c) =\\
= \tau^*b^{r-k-2} w_{k+2} (TN_\Sigma-\tau^*{f|_\Si}^*TQ).
\end{multline*}
If additionally $\de = 0$ (thus $c=0$), then similarly we have  
\begin{multline*}
w_r(TN_\Sigma-\tau^*{f|_\Si}^*TQ) = 
\sum_{j=0}^{k+1} w_j(\psi) \tau^*b^{r-j} = \tau^*b^{r-k-1}
w_{k+1} (TN_\Sigma-\tau^*{f|_\Si}^*TQ)
\end{multline*}
for $r \geq
k+1$.
These two equalities prove (1).

\par

Consider now a product $\prod_{i=1}^m w_{j_i}(T{N_{\Si}}-\tau^*{f|_\Si}^*TQ)$
which contains no instances of $w_1(T{N_{\Si}}-\tau^*{f|_\Si}^*TQ), \ldots,
w_{k+\de}(T{N_{\Si}}-\tau^*{f|_\Si}^*TQ)$. By  (1), it has the form
$$
\tau^*b^{n-(k+1+\de)m}
w_{k+1+\de}^{m}(T{N_{\Si}}-\tau^*{f|_\Si}^*TQ).
$$
In particular, this expression depends only on $m$ and hence any two
characteristic numbers of $w(T{N_{\Si}}-\tau^*{f|_\Si}^*TQ)$ with the same number
of multiplicands and no instances of $w_1, \dots, w_{k+\de}$ are equal. This
proves (2).

\par

By Theorem~\ref{perkov} and
Remark~\ref{foldpertkov} the formal difference bundle $T\bls \csi M-\pi^*f^*TQ$ is
stably equivalent to a $(k+1+\de)$-dimensional bundle. Therefore the
characteristic classes 
$w_r(T \bls \csi {M}-\pi^*f^*TQ)$ of $T \bls \csi {M}-\pi^*f^*TQ$ with $r>k+1+\de$ vanish.
Thus, by Remark~\ref{mapkobord}
those characteristic numbers of the virtual normal bundles
of the maps $N_\Si \to Q$ and $M \to Q$ which  contain $w_r$ with $r>k+1+\delta$ coincide.
This finishes the proof of (3).
\end{proof}

\begin{cor}\label{spin}
Let $w_1(TM), \ldots, w_k(TM) = 0$
and $Q$ be stably parallelizable.
 If there exists a fold map $M^n \to Q^{n-k}$, then
each of the nonzero characteristic numbers of $M$, which has more than one 
multiplicand, 
is equal to a number of the form $w_{k+1}^l w_{n-(k+1)l} [M]$ with $0 \leq
l \leq \frac{n}{k+1}-1$.
\end{cor}

\subsubsection{Adding Dold relations to the relations of Proposition~\ref{karosztalyok}}
In this section, we work in the case of fixed $n \geq 2$, $k=1$ and assume an
orientable source manifold $M^n$.

\par

Denote by $\mathcal I$ the linear space in the graded $\Z_2$-algebra 
$\Z_2[w_1,\dots,w_n, \dots]$ spanned by the set
\begin{multline*}
\{ q_1 - q_2 \in \Z_2[w_1,\ldots,w_n]_{{\mathrm {deg}} = n} :
q_1, q_2 \text{ are monomials in }\Z_2[w_2,\ldots,w_n]_{{\mathrm {deg}} = n},\\ |q_1| = |q_2|  \} \cup 
\{ q_1 \in \left[w_1\Z_2[w_1,\ldots,w_n]\right]_{{\mathrm {deg}} = n} \},
\end{multline*}
where $\vert q \vert$ denotes the length of the monomial $q$, i.e.\ the
number of (not necessarily different) indeterminants whose product is $q$.
Proposition~\ref{karosztalyok} (3) states that if $M$ is an $n$-manifold
admitting a codimension $-1$ fold map 
into a stably
parallelizable target, then the evaluation $w_i = w_i(TM)$ sends all the
members of $\mathcal I$ to $0$.

\par

Denote by $\mathcal D$ the linear
space spanned in $\Z_2[w_1,\dots,w_n, \ldots]$ by the set
$$\{ q \in \Z_2[w_1,\dots,w_n] :  p \in
\Z_2[w_1,\dots,w_n, \dots]_{{\mathrm {deg}} \leq n-1}, q = [w^{-1} 
\cdot \operatorname{Sq} p]_{{\mathrm {deg}}=n} \},$$
where $w$ stands for  the total Stiefel-Whitney class
$1 + w_1 +\cdots +w_n+ \cdots$.
We will apply a result of Dold \cite{Do} which states that all the
relations between the characteristic numbers of $n$-manifolds are 
exactly those of the form $q=0$ for $q \in \mathcal D$. Combining this set
of relations with $\{ q=0 : q \in \mathcal I_{} \}$ and proving that
 $\dim \Z_2[w_1,\dots, w_n]_{{\mathrm {deg}} = n}/
 (\mathcal D \+ \mathcal I_{}) = 0$ unless $n = 2^a + 2^b -1$ for some $a > b \geq 0$
 forms the core of the proof of Theorem~\ref{nullcob}.

\par

To utilize the relations obtained in Proposition~\ref{karosztalyok} (3), we
will consider the graded algebra homomorphism
$$
\varrho \co \Z_2[w_1,\dots, w_n, \dots] \to \Z_2[x,t], \text{ } \operatorname{deg} x = 2, \operatorname{deg}
t = 1
$$
$$
\varrho(w_1)=0, \varrho(w_s)=xt^{s-2} \text{ for } s \geq 2.
$$
Define $\im \varrho_n$ to be  $\Z_2[x,t]_{{\mathrm {deg}} = n} \cap \operatorname{im}
\varrho = \langle xt^{n-2}, \dots, x^{\lfloor n/2 \rfloor}t^{n-2\lfloor n/2
\rfloor} \rangle$. It is straightforward to see that
$\ker \varrho \cap \Z_2[w_1,\dots]_{{\mathrm {deg}}=n}$ is exactly $\mathcal
I$, therefore
\begin{multline*}
\dim \Z_2[w_1,\dots]_{{\mathrm {deg}} = n}/(\mathcal D \+ 
\mathcal I_{}) = \dim \left(
\Z_2[w_1,\dots]_{{\mathrm {deg}} = n}/\mathcal I_{} \right)/
\left( \mathcal D/\mathcal D\cap\mathcal I_{} \right) =\\
= \dim \left( \Z_2[w_1,\dots]_{{\mathrm {deg}} = n}/\mathcal I_{} \right)
- \dim \im \varrho|_{\mathcal D} 
= \dim \im \varrho_n - \dim \im \varrho|_{\mathcal D}.
\end{multline*}
To calculate the image of $\mathcal D$  under $\varrho$, we use the Wu
formulas. For $u \geq 2$
\begin{align*}
\varrho \circ \operatorname{Sq} w_u &= \varrho \left( \sum_{d=0}^{u} \sum_{j=0}^{d}
\binom{u-d+j-1}{j} w_{u+j}w_{d-j} \right) =\cr
&= \sum_{d=0}^{u} \left( \binom{u-1}{d} xt^{u+d-2} + \sum^{d-2}_{j=0}
\binom{u-d+j-1}{j} x^2 t^{u+d-4} \right) =\cr
&= \sum_{d=0}^{u} \left( \binom{u-1}{d} xt^{u+d-2} + \binom{u-2}{d-2} x^2
t^{u+d-4} \right) =\cr
&= xt^{u-2}(t+1)^{u-1}+x^2t^{u-2}(t+1)^{u-2} = xt^{u-2}(t+1)^{u-2}(x+t+1).
\end{align*}
Similarly,
$$
\varrho(w^{-1}) = (1+x(1+t+t^2+\dots))^{-1} = \frac{t+1}{x+t+1},
$$
hence for $s,m \geq 0$, $s+2m+2 \leq n-1$ and a monomial $p \in  
\Z_2[w_2,\ldots,w_{n-1}]$ of degree $2m+2+s$ and length $\vert p \vert =
m+1$ the corresponding element $[w^{-1} \cdot \operatorname{Sq} p]_{{\mathrm {deg}}=n}$ of $\mathcal D$ is mapped by $\varrho$ to
\begin{align*}
R(s, m) := \varrho( [w^{-1}\operatorname{Sq}p]_{{\mathrm {deg}}=n} )&= \left[ \frac{t+1}{x+t+1} x^{m+1} t^s
(t+1)^s (x+t+1)^{m+1} \right]_{{\mathrm {deg}}=n} =\\
&= \left[ (t+1)^{s+1} t^s (x+t+1)^m x^{m+1} \right]_{{\mathrm {deg}}=n}.
\end{align*}
Note that if a monomial $p$ is divisible by $w_1$, then $\operatorname{Sq}
p$ and $[w^{-1}\operatorname{Sq}p]_{{\mathrm {deg}}=n}$ are also divisible
by $w_1$, consequently $\varrho( [w^{-1}\operatorname{Sq}p]_{{\mathrm
{deg}}=n} ) = 0$.

\par

Separating the expression for $R(s,m)$ by degree of $x$ we get
$$R(s,m) = \left[ \sum_{i=0}^m \binom{m}{i} x^{m+1+i} t^s (t+1)^{s+1+m-i} \right]_{{\mathrm {deg}}=n}
=$$
$$= \sum_{i=0}^m \binom{m}{i} \binom{s+1+m-i}{n-2m-2i-2-s} x^{m+1+i}
t^{n-2m-2i-2}.$$
Recall that we use the convention that $\binom{\al}{\be}=0$ if $\be < 0$ or
$\al < \be$. Note that the binomial coefficient in the above sum is equal to $0$ if
the exponent of $t$ is negative.
When $p$ is the constant $1$, 
we have 
\begin{align*}
R_0 &:= \varrho( [w^{-1}\operatorname{Sq} 1 ]_{{\mathrm {deg}}=n} ) = \left[ \frac{t+1}{x+t+1} \right]_{\operatorname{deg} = n} = \left[
\frac{1}{1+\frac{x}{1+t}} \right]_{\operatorname{deg} = n} =\\
&= \sum_{1 \leq j \leq n/2} x^j \left[ (1+t)^{-j} \right]_{\operatorname{deg} =
n-2j} = \sum_{1 \leq j \leq n/2} \binom{n-j-1}{n-2j} x^j t^{n-2j}.
\end{align*}

\par

Let $V_R$ denote the set $\{ (s,m) : s,m \geq 0, s+2m+2 \leq n-1 \}$.
Therefore $\varrho (\mathcal D)$ is equal to the linear span of the
set $\{ R(s,m) : (s,m) \in V_R \} \cup \{ R_0\}$
in $\Z_2[x,t]$. Denote the linear span of
$\{ R(s,m) : (s,m) \in V_R \}$ by $\mathcal R_+$
and denote $\varrho (\mathcal D)$ by $\mathcal R$.
\subsubsection{The dimension of  $\mathcal R$}

The space $\mathcal R$ is contained in $\im \varrho_n$, that is, $\mathcal R$
is contained in
$\langle xt^{n-2}, \dots, x^{\lfloor n/2 \rfloor}t^{n-2\lfloor n/2
\rfloor} \rangle$. We will check whether the monomials $xt^{n-2}$, $\dots$,
$x^{\lfloor n/2 \rfloor}t^{n-2\lfloor n/2 \rfloor}$ are contained in
$\mathcal R$ separately in the cases of odd and even $n$. We will use the
criterion of \cite{Gla}
cited above, which states that $\binom{b}{a}$, $0 \leq a \leq b$, is
even if and only if there is a binary position at which $a$ has the digit
$1$ and $b$ has the digit $0$.
\begin{lem}\label{binomlem}
The binomial coefficients $\binom{K-p}{p}$ with $0< p \leq \frac{K}{2}$ are
all even if and only if $K+1$ is a power of $2$.
\end{lem}
\begin{proof}
If $K+1$ is a power of $2$, then $K$ written in binary contains only digits
$1$, hence $p$ and $K-p$ are complementary to each other. Since $p \neq 0$, 
there is a digit $1$ in its binary representation, thus in
the same position $K-p$ has digit $0$ and the criterion of \cite{Gla}
implies that $\binom{K-p}{p}$ is even.
Conversely, if $K$ contains the bit pattern $...10...$ at position $h$, say,
then $\binom{K-2^h}{2^h}$ is odd by the same criterion.
\end{proof}
Due to our convention, this result implies that  $K+1$ is a power of $2$ 
if and only if
the binomial coefficients $\binom{K-p}{p}$ are even for all $p > 0$.
\subsubsection{Case of $n$ even}
For $n =2$, we have $R_0 = x$, hence 
$\im \varrho_n = \langle x \rangle = \mathcal R$.

\par

For $n >2$, note that the  monomial $xt^{n-2}$ occurs as 
a summand in $R(s,m)$ only in the case
$m=0$ and  $R(s,0)=\binom{s+1}{n-2-s} xt^{n-2}$.
If $n \geq 3$, then for $s = n-3$ we have $0 < n-2-s
\leq (n-1)/2$
and $(s,0) \in V_R$.
 If $n$ is not a power of $2$, then we apply Lemma~\ref{binomlem} with $K = n-1$ and $p= n-2-s$,
and obtain that the coefficient of $xt^{n-2}$ in $R(s,0)$ is not $0$. If $n$ is a power of
$2$, then Lemma~\ref{binomlem} with the same choice of $K = n-1$ and $p =
n-2-s$ implies that $R(s,0)=0$ for all $(s,0) \in V_R$. Hence $xt^{n-2} \in \mathcal R_+$ if and only if $n$ is not a power
of $2$. Note that if $xt^{n-2} \not\in \mathcal R_+$, then $xt^{n-2}$ does not
appear as a summand in any elements of $\mathcal R_+$.

\par

If $n=4$, then $V_R = \{ (0,0), (1, 0) \}$. We have
$R(0,0) = 0$ as one can check easily and
above we showed that $R(1,0) = 0$, thus $\mathcal R_+$ consists only of the zero
element. 

\par  

Next consider $m=1, 3, \dots, \frac{n-4}{2}$ for $n \geq 4$ if $4 \nmid n$
and $m=1, 3, \dots, \frac{n-6}{2}$ for $n \geq 6$ if $4\mid n$.
Choosing $s=n-2m-3$ gives us $(s,m) \in V_R$ and $R(n-2m-3,m) = \binom{m}{0} \binom{n-m-2}{1}
x^{m+1}t^{n-2m-2}$. Since $n-m-2$ is odd, $R(n-2m-3,m) =
x^{m+1}t^{n-2m-2}$. Therefore $x^{m+1}t^{n-2m-2}$ is in $\mathcal R_+$.
Setting $s=n-2m-4$ gives $(s,m) \in V_R$ and
$$
R(n-2m-4,m) =  \binom{m}{0} \binom{n-m-3}{2} x^{m+1}t^{n-2m-2} +
\binom{m}{1} \binom{n-m-4}{0} x^{m+2}t^{n-2m-4}.
$$
The first summand is in $\mathcal R_+$ by the argument above, thus so
is the second one, which equals to $x^{m+2}t^{n-2m-4}$ due to $m$ being odd.
Therefore we obtain that for $n \geq 4$, $4 \nmid n$ the monomials $x^2
t^{n-4}, \dots, x^{n/2}$ are in $\mathcal R_+$, and for $n \geq 6$, $4 \mid
n$ the monomials $x^2 t^{n-4}, \dots, x^{(n-2)/2} t^2$ are in $\mathcal
R_+$.

\par

The only monomial not covered by the cases detailed above is $x^{n/2}$ in
the case when $n$ is divisible by $4$ and $n \geq 6$. Since all the other monomials either
belong to $\mathcal R_+$ or do not appear as summands in any $R(s,m)$ for $(s,m) \in V_R$, we have that
$x^{n/2} \in \mathcal R_+$ if and only if $x^{n/2}$ occurs as a summand in
an $R(s,m)$, $(s,m) \in V_R$. The coefficient of $x^{n/2}$ in $R(s,m)$ 
can  be nonzero only when
$n=2m+2i+2$ for some $0 \leq i \leq m$ and $s=0$, and then 
the coefficient is
$$
\binom{m}{i} \binom{s+1+m-i}{n-2m-2i-2-s} = \binom{m}{\frac{n}{2}-m-1}
\binom{2m+2-\frac{n}{2}}{0} = \binom{m}{\frac{n}{2}-m-1},
$$
and we have $(s,m) \in V_R$ if and only if $m \geq i \geq 1$, $n \geq 6$.
By Lemma~\ref{binomlem}, $\binom{m}{\frac{n}{2}-m-1}$ is even for all possible $m$
exactly when $\frac{n}{2}$ is a power of $2$.

\par

To summarize, when $n$ is even, the set $\{ R(s,m) : (s,m) \in V_R \}$ generates the space
$\im \varrho_n = \langle
xt^{n-2}, \dots, x^{n/2} \rangle$
if $n$ is not a power of $2$.
If $n$ is a power of $2$, then we know that $\mathcal R_+$ is spanned by all
monomials in $\im \varrho$ of degree $n$ except for $xt^{n-2}$ and
$x^{n/2}$. Let us check the coefficients of $xt^{n-2}$ and $x^{n/2}$ in $R_0
= \binom{n-2}{n-2} xt^{n-2} + \dots + \binom{n-\frac{n}{2}-1}{0} x^{n/2}$.
Both of their coefficients are $1$, hence $\mathcal R = \langle \mathcal R_+
\cup \{ R_0 \} \rangle$ has codimension $1$ in $\im \varrho_n$ and 
 $\im \varrho_n / \mathcal R$ is spanned by $xt^{n-2} + \mathcal R =
x^{n/2} + \mathcal R$.

\subsubsection{Case of $n$ odd}

 Let us call a  monomial $x^h t^{n-2h}$ {\it admissible} if
$h$ is not a power of $2$.
\begin{lem}\label{rek}
For an admissible  monomial $x^h t^{n-2h}$ with $1 \leq h \leq \frac{n-1}{2}$ and 
$2^u < h < 2^{u+1}$, where $u \geq 0$, there exist an integer
$r(h)$, a set of integers $E_h$ 
with $2^u \leq \al \leq h-1$ for all $\al \in E_h$,
and an element
 $R_h \in \mathcal R_+$ such that 
 $R_h + x^h t^{n-2h}$
is a linear combination of monomials $x^{\al} t^{n-2\al}$, $\al \in E_h$.
\end{lem}
\begin{proof}
Take the greatest $r = r(h) \geq 0$ for which $2^r \mid h$. 
Note that 
\begin{equation}\label{alsokorl}
h - 1, \ldots, h -2^{r} \geq 2^u
\end{equation}
since $h-2^r$ has the same binary form as $h$ except for the least
significant digit $1$, which is changed to $0$.
Also note that $h
\equiv 2^r \operatorname{ mod } 2^{r+1}$ and $h \geq 2^{r+1} +2^r$. 
Consider $R(n-2h,h-1-2^r)$. This polynomial is in $\mathcal R_+$ since
$(n-2h,h-1-2^r) \in V_R$.
Indeed,  $h-1 \geq 2^{r+1} >
2^r$, $2h<n$ and $n-2h+2(h-1-2^r)+2 = n-2^{r+1} < n$. We have
\begin{multline*}
R(n-2h,h-1-2^r) = \binom{h-1-2^r}{0} \binom{n-h-2^r}{2^{r+1}}
x^{h-2^r}t^{n-2h+2^{r+1}} + \cdots \\
  \cdots +\binom{h-1-2^r}{2^r} \binom{n-h-2^{r+1}}{0} x^h t^{n-2h}.
\end{multline*}
The coefficient of the last monomial of this sum is nonzero because
$h \geq 2^{r+1}+2^r$, $n-h-2^{r+1} > n-2h > 0$, and $h-2^r-1
\equiv -1 \operatorname{ mod } 2^{r+1}$ implies that the $r$-th binary digit
of $h-2^r-1$ is $1$. 

Let $R_h$ be $R(n-2h,h-1-2^r)$
and let $E_h$ be $\{h - 1, \ldots, h -2^{r} \}$.
 Then, by (\ref{alsokorl}) we have the statement.
\end{proof}

\begin{prop}\label{thm:power2}
Let $1 \leq h \leq \frac{n-1}{2}$ and assume that $u \geq 0$ is the greatest
integer such that $2^u \leq h$. Then $x^h t^{n-2h} \in \mathcal R_+$ or $x^h
t^{n-2h} + x^{2^u}t^{n-2^{u+1}} \in \mathcal R_+$ holds.
\end{prop}
\begin{proof}
If $h$ is a power of $2$, then the statement obviously holds.
Hence we can assume that $x^h t^{n-2h}$ is admissible  and $2^u < h < 2^{u+1}$.
Apply Lemma~\ref{rek} to $x^h t^{n-2h}$, then
 $x^{h} t^{n-2h} + R_h$ is a 
 linear combination of  monomials, where the exponents $\al \in E_h$ 
 satisfy $2^u \leq \al \leq h-1$. Let $E'_h = 
 \{ \al \in  E_h: \al > 2^u \}$.
 
  Again, if  $h-1 > 2^u$ and $E'_h \neq \emptyset$, then apply Lemma~\ref{rek} 
 to the admissible monomials of the linear combination $x^{h} t^{n-2h} + R_h$,
 then we obtain  that $x^{h} t^{n-2h} + R_h + 
 \sum_{\al \in E'_h} R_{\al}$ is a linear combination
 of monomials whose degree in $x$ is at least $2^u$ and smaller than $h-1$.
 Again, if $h-2 > 2^u$ and there
 are resulting admissible monomials in the last linear combination,
  then apply Lemma~\ref{rek}, and
 iterate this procedure until 
 we get that $x^{h} t^{n-2h}  + \tilde R = \ep x^{2^u} t^{n-2^{u+1}}$,
 where $\tilde R \in \mathcal R_+$ and $\ep \in \{0,1\}$.
 Note that the procedure finishes in a finite number of steps since at each step 
 the linear combination of the next step has smaller degrees  in $x$, while a common
 lower limit for the degrees of $x$ is $2^u$.
 This proves our claim.
\end{proof}

\begin{prop}\label{thm:choice}
Let $1 \leq h \leq \frac{n-1}{2}$.
For every $r \geq 0$, if
$x^h t^{n-2h} \not\in \mathcal R_+$ and
$n-2h \geq 2^r-1$, then
$2^r | n-h+1$.
\end{prop}
\begin{proof} The proof will proceed by induction on $r$, with $r=0$ as the
trivial starting case: $1|n-h+1$ always holds. 

\par

Let $r \geq 1$ and suppose that the statement holds for $r-1$. Let $h$ be
such that $x^h t^{n-2h} \not\in \mathcal R_+$ and $n-2h \geq 2^r-1$. Assume
indirectly that $2^r \nmid n-h+1$. 
We have $n-2h \geq 2^r-1 > 2^{r-1} -1$ hence
by the induction hypothesis we have
$n-h+1 \equiv 2^{r-1} \mod 2^r$.

\par       

Consider 
\begin{multline}\label{osszeg}
R(n-2h-2^{r-1},h-1) = \binom{h-1}{0}
\binom{n-h-2^{r-1}}{2^{r-1}} x^h t^{n-2h} + \cdots \\  \cdots + \binom{h-1}{\lfloor
2^{r-2} \rfloor} \binom{n-2h-2^{r-2}}{2^{r-1} - 2\lfloor 2^{r-2} \rfloor}
x^{h+\lfloor 2^{r-2} \rfloor} t^{n-2h-2\lfloor 2^{r-2} \rfloor},
\end{multline}
where taking the integral part of $2^{r-2}$ is only needed to handle the
case of $r=1$. 
Since $h \geq 1$, $n-2h-2^{r-1} \geq 2^r-1-2^{r-1} =
2^{r-1}-1 \geq 0$ and $n-2h-2^{r-1}+2(h-1)+2 = n-2^{r-1} < n$,
we have that $R(n-2h-2^{r-1},h-1) \in \mathcal R_+$.

\par

For $y = h+1, \ldots, h+\lfloor 2^{r-2} \rfloor$ we have
$n-2y \geq n-2(h+\lfloor 2^{r-2} \rfloor) \geq 2^r - 1-2^{r-1} = 2^{r-1} -1$, and
since by the induction hypothesis $2^{r-1} \mid n-h+1$,
none of the values $n-y+1$ can be
divisible by $2^{r-1}$. Applying the induction hypothesis again gives us
that all of the monomials in \eqref{osszeg} except possibly the first one
are in $\mathcal R_+$. But the coefficient of the first term is nonzero:
$\binom{h-1}{0} = 1$ and $\binom{n-h-2^{r-1}}{2^{r-1}} = 1$ by \cite{Gla}
since $n-h-2^{r-1} \equiv 2^r-1 \operatorname{ mod } 2^r$ has the binary
digit $1$ at the only location where $2^{r-1}$ has a $1$. Therefore
$R(n-2h-2^{r-1},h-1)+x^ht^{n-2h} \in \mathcal R_+$ 
and consequently $x^hy^{n-2h} \in \mathcal R_+$, finishing the
proof.
\end{proof}

We apply Proposition \ref{thm:choice} to monomials of the form $x^{2^u}
t^{n-2^{u+1}}$, $u \geq 0$, $n$ is odd and $2^{u+1} \leq n$, with the choice of 
 $r \geq 0$  so that $2^{r+1} < n < 2^{r+2}$. Note that
 $u \leq r$ due to $2^u
\leq \frac{n-1}{2} < 2^{r+1}$.  We get that if $x^{2^u} t^{n-2^{u+1}} \not\in \mathcal
R_+$, then at least one of the following has to hold:
\begin{enumerate}[\rm (a)]
\item $2^r \mid n-2^u+1$. We know that $n-2^u$ is at least $n-2^r > 2^{r+1} - 2^r =
2^r$, and on the other hand $n-2^u$ is at most $n-1 < 2^{r+2} -1$. There are
only two integers $i$
 in the open interval $(2^r, 2^{r+2}-1)$  which satisfy
the divisibility condition $2^r | i+1$, namely $2^{r+1}-1$ and $3 \cdot 2^r-1$. Hence $n-2^u$ is either $2^{r+1}-1$ or $3 \cdot 2^r-1$.
\item $n-2^{u+1} < 2^r-1$. Then  $2^{u+1} > n-2^r+1 > 2^r$ 
since $2^{r+1} < n$, and $u \leq r$ implies that $u=r$.
\end{enumerate}

We claim that in the case $n-2^u = 3 \cdot 2^r-1$ of (a) the monomial
$x^{2^u}t^{n-2^{u+1}}$ is actually in $\mathcal R_+$. Indeed, check the
statement of Proposition \ref{thm:choice} for $r+1$. Then $2^{r+1} \nmid n-2^u+1 = 3
\cdot 2^r$, and $n-2^{u+1} = 3 \cdot 2^r - 1 - 2^u \geq 2^{r+1} - 1$
since $2^u \leq 2^r$, therefore we have 
 $x^{2^u}t^{n-2^{u+1}} \in \mathcal R_+$.

\par       

Hence if $x^{2^u} t^{n-2^{u+1}} \not\in \mathcal R$, then we are left with
two possibilities:
\begin{enumerate}[\rm (a)]
\item
$n=2^{r+1}+2^u-1$,
\item
$n-2^{u+1} < 2^u-1$ and
$2^{u+1} < n < 2^{u+2}$. 
\end{enumerate}

In the case (b), note that if $x^{2^u}t^{n-2^{u+1}} \not\in \mathcal
R_+$, then Proposition \ref{thm:choice} implies that for any $0 \leq r' \leq u$ either
\begin{enumerate}[\rm (i)]
\item
$n-2^{u+1} < 2^{r'}-1$ or 
\item
$n-2^u \equiv 2^{r'} -1  \mod 2^{r'}$. Due to ${r'} \leq u$ this condition is
equivalent to $n-2^{u+1} \equiv 2^{r'}-1 \mod 2^{r'}$.
\end{enumerate}

In the case (b) choose $r'$ to satisfy $2^{r'} \leq n-2^{u+1} < 2^{r'+1}$. This value of $r'$ will
be smaller than $u$ due to $n -2^{u+1} < 2^u -1$. For this choice of $r'$, condition (i) fails, thus
condition (ii) has to hold. This implies that $n-2^{u+1}-(2^{r'}-1) = l 2^{r'}$.
 This integer $l$ can be only $1$ because $n-2^{u+1} < 2^{r'+1}$.
Thus, $n-2^{u+1} =
2^{r'+1}-1$.

\par

Therefore if $x^{2^u} t^{n-2^{u+1}} \not\in \mathcal
R_+$, then we have two possible cases:
\begin{enumerate}[(a)]
\item
$n=2^{r+1}+2^u-1$, where we chose $r \geq 0$ 
 so that $2^{r+1} < n < 2^{r+2}$, this implied $u \leq r$,
\item
$n = 2^{u+1} + 
2^{r'+1}-1$, 
where we chose $r' \geq 0$  so that $2^{r'} \leq n-2^{u+1} < 2^{r'+1}$,
this implied $r' < u$. 
\end{enumerate}

By Proposition~\ref{thm:power2} 
in both cases (a) and (b) we have $\mathcal R_+ = \im \varrho_n$, unless
there are positive integers $a > b$ such that $n=2^a+2^b-1$. Moreover in these
exceptional cases, when $n=2^a+2^b-1$ with $a>b>0$, the linear space
$\mathcal R_+$ has to contain all the monomials $x^{2^u} t^{n-2^{u+1}}$
except possibly those with $u=a-1$ (in case (b)) and $u=b$ (in case (a)).

\begin{prop}
If $n=2^a+2^b-1$ with $a>b>0$ and $u=a-1$ or $u=b$, then the monomial
$x^{2^u} t^{n-2^{u+1}}$ does not appear as a summand in any $R(s,m) \in \mathcal R_+$.
\end{prop}
\begin{proof}
In the relation $R(s,m)$ the monomial $x^{2^u}t^{n-2^{u+1}}$ has the
coefficient 
\begin{multline}\label{allofthese}
\binom{m}{2^u-m-1} \binom{m+s+1-(2^u-m-1)}{n-2m-2(2^u-m-1)-2-s}
= \\ 
=\binom{m}{2^u-1-m} \binom{2m+2-2^u+s}{n-2^{u+1}-s}.
\end{multline}
The first binomial coefficient is even unless $2^u=m+1$ according to
Lemma~\ref{binomlem} for $K=2^u-1$, hence we only consider the case $m=2^u-1$.
Then the
second binomial coefficient becomes $\binom{2^u+s}{n-2^{u+1}-s}$ with $s$
running from $0$ to $n-2m-3=n-2^{u+1}-1$
as follows from the condition $(s,m) \in V_R$.
 For $u=b$ we have $n-2^u=2^a-1$
and Lemma~\ref{binomlem} with $K=n-2^u$ implies that all of these
coefficients (\ref{allofthese}) are even. For $u=a-1$ we have $2^{u+1}+s \geq 2^{u+1} = 2^{a}$ and
$n-2^{u+1}-s \leq n-2^a = 2^b-1 < 2^{a-1}$, thus the criterion of \cite{Gla} gives the same
results for $\binom{2^u+s}{n-2^{u+1}-s}$ and
$\binom{2^u+s-2^{a-1}}{n-2^{u+1}-s} = \binom{s}{2^b-1-s}$. Since
$n-2^{u+1}-s \geq 1$, Lemma~\ref{binomlem} proves that $\binom{s}{2^b-1-s}$ is
even for all choices of $s$, as claimed.
\end{proof}

This means that if $n=2^a+2^b-1$, $a>b>0$, then the monomials 
$x^{2^{a-1}} t^{n-2^{a}}$ and $x^{2^b} t^{n-2^{b+1}}$
never appear as summands in any $R(s,m) \in \mathcal R_+$ and hence the algorithm of
Proposition~\ref{thm:power2}  leads to  $x^ht^{n-2h} \in \mathcal
R_+$ when $u=a-1$ or $u=b$.
Consequently, $\mathcal R_+$ is spanned by all monomials from $xt^{n-2}$ to
$x^{\frac{n-1}{2}} t$ except $x^{2^{a-1}}t^{n-2^a}$ and
$x^{2^b}t^{n-2^{b+1}}$, which span a linear space complementary to $\mathcal
R_+$.

\par       

To summarize, when $n$ is odd, then we have three possibilities:
\begin{itemize}
\item If $n \neq 2^a+2^b-1$ for any $a>b>0$, then $\mathcal R_+ = \im
\varrho_n$.
\item If $n = 3 \cdot 2^b-1$ for some $b>0$, then
$$\mathcal R_+ = \langle xt^{n-2}, \dots, x^{2^b-1}t^{n-2^{b+1}+2},
x^{2^b+1}t^{n-2^{b+1}-2}, \dots, x^{\frac{n-1}{2}} t \rangle.$$
\item If $n = 2^a+2^b-1$ for some $a>b+1$, $b>0$, then
\begin{align*}
\mathcal R_+ = \langle xt^{n-2}, &\dots , x^{2^b-1}t^{n-2^{b+1}+2},
x^{2^b+1}t^{n-2^{b+1}-2}, \dots,\\
& x^{2^{a-1}-1}t^{n-2^a+2}, x^{2^{a-1}+1}t^{n-2^a-2},\dots,
x^{\frac{n-1}{2}} t \rangle.
\end{align*}
\end{itemize}
 Finally, the relation $R_0$ contains the monomials $x^{2^{a-1}}
t^{n-2^a}$ and $x^{2^b} t^{n-2^{b+1}}$ with the coefficients
$\binom{n-2^{a-1}-1}{n-2^a}= \binom{2^b+2^{a-1}-2}{2^b-1}$ and
$\binom{n-2^b-1}{n-2^{b+1}}= \binom{2^a-2}{2^a-2^b-1}$,
both of which are of the form $\binom{{\mathrm {even}}}{{\mathrm {odd}}}$ and thus are  even by
the criterion of \cite{Gla}. Hence $\mathcal R_+ = \mathcal R$.

\subsubsection{Proofs of the statements}

\begin{proof}[Proof of Theorem~\ref{nullcob}]
By the above, if $n \neq 2^a + 2^b -1$, $a>b \geq 0$, then any oriented
$n$-manifold which has a fold map in codimension $-1$ is unoriented
null-cobordant. Unless $n$ is divisible by $4$, this implies that $M$ is also
oriented null-cobordant, see \cite{Wa}.

\par

In the case of $n = 2^a + 2^b -1$, $a>b \geq 0$, the Stiefel-Whitney
characteristic numbers of $M$ which belong to the complete preimage
$\varrho^{-1}(\mathcal R)$ have to vanish. This leaves the following
possibilities for nonzero characteristic numbers:
\begin{itemize}
\item if $n$ is a power of $2$, then $\varrho^{-1}(\mathcal R)$ is spanned by
$w_n+w_2^{n/2}$ and all monomials except $w_n$ and $w_2^{n/2}$. In this case
$[M] \in \mathfrak A^1$.
\item if $n=2^a+2^b-1$, $a>b>0$, then $\varrho^{-1}(\mathcal R)$ is spanned by
all monomials of length not equal to either $2^b$ or $2^{a-1}$ as well as
the relations in $\mathcal I$ corresponding to these exceptional lengths.
When $a=b+1$, the two lengths coincide and we get that $[M] \in \mathfrak
B^1$, while in the other case we get that $[M] \in \mathfrak C^2$.
\end{itemize}

\par

In the remaining case of $n$ divisible by $4$, we need to additionally
calculate the Pontryagin characteristic numbers of $M$ to determine its
oriented cobordism class. Theorem~\ref{szingtusk} shows that all the rational 
Pontryagin classes of $M$ except $p_1^{\Q}(TM)$ have to vanish, hence the only
Pontryagin number that may be nonzero is $p_1^{n/4}[M]$. If we additionally
assume that $M$ is unoriented null-cobordant, then this number has to be
even as its reduction modulo $2$ is the Stiefel-Whitney characteristic
number $w_2^{n/2}[M]$.
\end{proof}
\begin{proof}[Proof of Theorem~\ref{nullcoborising}]
By Proposition~\ref{Kgamma} we know that $TM \+ \varepsilon^1 = \zeta^2 \+
\varepsilon^{n-1}$ for some $2$-dimensional bundle $\zeta$. Hence
$w(TM)=1+w_1(\zeta)+w_2(\zeta)$, and we will denote the characteristic class
$w_i(\zeta)$ by $w_i$ for brevity. The only nonzero total Steenrod squares
of these classes are $Sq(w_1)=w_1(1+w_1)$ and $Sq(w_2)=w_2+w_2w_1+w_3+w_2^2=
w_2(1+w_1+w_2)$. Thus, we can compute the Dold relation corresponding to
the polynomial $w_1^aw_2^b$ with $a+2b \leq n-1$, $a,b \geq 0$, as the
degree $n$ part of the expression
$$
\frac{w_1^a(1+w_1)^aw_2^b(1+w_1+w_2)^b}{1+w_1+w_2} =
w_1^a(1+w_1)^aw_2^b(1+w_1+w_2)^{b-1}.
$$
Setting $b=0$, $0 \leq a \leq n-1$ gives
\begin{multline*}
R(a) := \left[ \frac{w_1^a(1+w_1)^a}{1+w_1+w_2}\right]_{{\mathrm {deg}}=n} = \left[
\frac{w_1^a(1+w_1)^{a-1}}{1+\frac{w_2}{1+w_1}}\right]_{{\mathrm {deg}}=n} = \\
= \left[ w_1^a(1+w_1)^{a-1} \sum_{j=0}^{\infty}\frac{w_2^j}{(1+w_1)^j}\right]_{{\mathrm {deg}}=n} =
\sum_{j=0}^{\frac{n-a}{2}} w_1^a w_2^j \left[ (1+w_1)^{a-1-j}\right]_{{\mathrm {deg}}=n-a-2j}
=\\= \sum_{j=0}^{\frac{n-a}{2}} \binom{a-1-j}{n-a-2j} w_1^{n-2j} w_2^j.
\end{multline*}
Here we use the analytical definition of binomial coefficients: $\binom{u}{v} = 0$ if $v<0$, $\binom{u}{0}
= 1$ and $\binom{u}{v}= \frac{u(u-1)\cdots(u-v + 1)}{v!}$ in the other
cases. Choosing $a=n-2m$ for any $0 < m \leq n/2$ we get
\begin{align*}
R(n-2m) &= \sum_{j=0}^{m} \binom{n-2m-1-j}{2m-2j} w_1^{n-2j} w_2^j
=\\
&= \binom{n-2m-1}{2m} w_1^n + \dots +\binom{n-3m-1}{0} w_1^{n-2m} w_2^m
\end{align*}
with analytical binomial coefficients.
Here the exponent of $w_2$ in all the summands except the last one is less
than $m$, and the last summand has coefficient $1$. Therefore for all $0 < m
\leq n/2$ the monomial $w_2^m w_1^{n-2m}$ is linearly dependent on the
monomials with smaller exponents of $w_2$ and $R(n-2m)$. Consequently, all
the monomials $w_2 w_1^{n-2}, \ldots, w_2^m w_1^{n-2m}$ are linearly
dependent on $w_1^n$ and $\{ R(n-2m) : 1 \leq m \leq n/2 \}$. Evaluating
these classes on the fundamental class of $M$, we get that all
Stiefel-Whitney characteristic numbers of $[M]$ depend linearly on
$w_1^n[M]$ (with coefficients depending only on $n$). This condition either
defines the $1$-dimensional linear subspace $\mathfrak D^1 \leq \mathfrak
N_n$ or implies that $M$ is unoriented null-cobordant.
\end{proof}

\begin{proof}[Proof of Proposition~\ref{nullcobnagykodimkonnyu}]
Choose any index set $I = (i_1, \ldots, i_r)$ such that $r= \vert I \vert
\geq 2$ and $\sum_{j=1}^{r} i_j = n$. If for any $j$ we have $i_j \leq k$,
then $w_I=0$ due to the vanishing condition imposed on the Stiefel-Whitney
classes of $TM$. If $i_j \geq k+1$ for
$j= 1, \ldots r$, then for $J = (n-(r-1)(k+1), k+1, \ldots, k+1)$ the
characteristic numbers $w_I[M]$ and $w_J[M]$ coincide by
Corollary~\ref{spin}. But $w_{k+1}(TM)=0$ by \cite[Problem~8-B]{MiSta}, implying that
$w_J[M] =0$ and thus $w_I[M] =0$ whenever $|I| \geq 2$.
\end{proof}

\begin{proof}[Proof of Proposition~\ref{nullcobnagykodimkonnyu2}]
Perturb the Morin map to get
a cusp map \cite{Sad}.
By Proposition~\ref{karosztalyok} and \cite[Problem~8-B]{MiSta} the statement follows 
similarly to the previous one, details are left to the reader.
\end{proof}

\end{document}